\documentclass[12pt]{amsart}

\usepackage{amsfonts, amssymb, amsthm, amsmath, braket, hyperref, mathtools, comment, mathrsfs, enumitem}
\usepackage[usenames,dvipsnames, table]{xcolor}
\usepackage[all]{xy}
\usepackage{fullpage}

\usepackage[nameinlink]{cleveref}

\numberwithin{equation}{subsection}

\newtheorem{theorem}[equation]{Theorem}
\newtheorem{corollary}[equation]{Corollary}
\newtheorem{cor}[equation]{Corollary}
\newtheorem{lem}[equation]{Lemma}
\newtheorem{lemma}[equation]{Lemma}
\newtheorem{prop}[equation]{Proposition}

\theoremstyle{definition}
\newtheorem{definition}[equation]{Definition}

\theoremstyle{remark}
\newtheorem{remark}[equation]{Remark}

\newcommand{\defi}[1]{\textsf{#1}} % for defined terms

\newcommand{\C}{\mathbb{C}}
\newcommand{\Q}{\mathbb{Q}}
\newcommand{\R}{\mathbb{R}}
\newcommand{\Z}{\mathbb{Z}}
\newcommand{\PP}{\mathbb{P}}

\newcommand{\bbC}{\mathbb{C}}

\newcommand{\bbQ}{\mathbb{Q}}
\newcommand{\bbR}{\mathbb{R}}
\newcommand{\bbZ}{\mathbb{Z}}

\newcommand{\calR}{\mathcal{R}}
\newcommand{\calT}{\mathcal{T}}

\newcommand{\scrE}{\mathscr{E}}
\newcommand{\scrO}{\mathscr{O}}

\newcommand{\Qalg}{\bbQ^{\textup{al}}}

\newcommand{\prm}{^\prime}
\newcommand{\parent}[1]{\left(#1\right)}

\newcommand{\inv}{^{-1}}

\newcommand{\abs}[1]{\left|#1\right|}
\DeclareMathOperator{\Gal}{Gal}
\DeclareMathOperator{\ord}{ord}
\DeclareMathOperator{\Nm}{Nm}
\DeclareMathOperator{\repart}{Re}

\DeclareMathOperator{\Area}{area}
\DeclareMathOperator{\len}{len}

\DeclareMathOperator{\lcm}{lcm}
\DeclareMathOperator{\disc}{disc}

\newcommand{\ceil}[1]{\left\lceil #1 \right\rceil}
\newcommand{\floor}[1]{\left\lfloor #1 \right\rfloor}

\newenvironment{enumalph}
{\begin{enumerate}}
	{\end{enumerate}}

\newenvironment{enumroman}
{\begin{enumerate}}
	{\end{enumerate}}

\usepackage{colonequals}
\usepackage{hyperref}
\hypersetup{colorlinks=true,urlcolor=blue,citecolor=blue,linkcolor=blue}

\renewcommand{\set}[1]{\left\{#1\right\}}

\DeclareMathOperator{\hht}{ht}

\DeclareMathOperator{\twistheight}{twht}
\DeclareMathOperator{\twht}{twht}
\newcommand{\rawheight}{H}

\DeclareMathOperator{\tors}{tors}
\DeclareMathOperator{\Res}{Res}
\DeclareMathOperator{\res}{res}

\newcommand{\psmod}[1]{~(\textup{\text{mod}}~{#1})}

\newcommand{\tA}{A_0}
\newcommand{\tB}{B_0}

\newcommand{\upperratio}{\kappa}

\DeclareMathOperator{\mindefect}{md}
\newcommand{\md}{\mindefect}
\DeclareMathOperator{\twistdefect}{tmd}
\newcommand{\tmd}{\twistdefect}

\newcommand{\twistEX}{\scrE^{{\rm tw}}_{\leq X}}
\newcommand{\twistE}{\scrE^{\textup{tw}}}

\newcommand{\twistN}{N_{}^{\textup{tw}}}
\newcommand{\twistNly}{N_{\leq y}^{\rm tw}}
\newcommand{\twistNgy}{N_{> y}^{\rm tw}}
\newcommand{\NQ}{N_{}}

\newcommand{\cM}{M}

\newcommand{\cT}{T}
\newcommand{\tcT}{\widetilde{T}}
\newcommand{\tcalT}{\widetilde{\calT}}

\newcommand{\LQ}{L}
\newcommand{\twisth}{h^{\rm tw}}
\newcommand{\hQ}{h}
\newcommand{\twistL}{L^{\rm tw}}
\newcommand{\twistLR}{L_{\rm{rem}}^{\rm tw}}

\newcommand{\jvtable}[3]
{\begin{equation} \label{#1}\addtocounter{equation}{1} \notag
\begin{gathered}
#2 \\
\text{
\parbox[c]{4.5in}{\centering Table \ref*{#1}: #3}}
\end{gathered}
\end{equation}
}

\newcommand{\jvfigure}[3]
{\begin{equation} \label{#1}\addtocounter{equation}{1} \notag
\begin{gathered}
#2 \\
\text{
\parbox[c]{4.5in}{\centering Figure \ref*{#1}: #3}}
\end{gathered}
\end{equation}
}

\newenvironment{enumalg}
{\begin{enumerate}}
{\end{enumerate}}

\begin{document}
\title[Counting elliptic curves with a 7-isogeny]{Counting elliptic curves over the rationals \\ with a 7-isogeny}

\author{Grant Molnar}
\address{Department of Mathematics\\
Dartmouth College\\
Hanover, NH 03755-3551}
\email{Grant.S.Molnar.GR@dartmouth.edu}
\urladdr{http://www.grantmolnar.com}

\author{John Voight}
\address{Department of Mathematics\\
Dartmouth College\\
Hanover, NH 03755-3551}
\email{jvoight@gmail.com}
\urladdr{http://www.math.dartmouth.edu/~jvoight}

\begin{abstract}
We count by height the number of elliptic curves over the rationals, both up to isomorphism over the rationals and over an algebraic closure thereof, that admit a cyclic isogeny of degree $7$.
\end{abstract}

\maketitle

\setcounter{tocdepth}{1}
\tableofcontents

\section{Introduction}\label{Section: Introduction}

\subsection{Motivation and setup}

Number theorists have an enduring, and recently renewed, interest in the arithmetic statistics of elliptic curves: broadly speaking, we study asymptotically the number of elliptic curves of bounded size with a given property. More precisely, every elliptic curve $E$ over $\Q$ is defined uniquely up to isomorphism by a Weierstrass equation of the form 
\begin{equation} 
E \colon y^2 = x^3 + Ax + B\label{Equation: Weierstrass equation}
\end{equation}
with $A,B \in \Z$ satisfying $4A^3+27B^2 \neq 0$ and such that no prime $\ell$ has $\ell^4 \mid A$ and $\ell^6 \mid B$. Let $\scrE$ be the set of elliptic curves of this form: we define the \defi{height} of $E \in \scrE$ by 
\begin{equation} 
\hht(E) \colonequals \max(\abs{4A^3},\abs{27B^2}). 
\end{equation}
For $X \geq 1$, let $\scrE_{\leq X} \colonequals \{E \in \scrE : \hht(E) \leq X\}$. Mathematicians have studied the count of those $E \in \scrE_{\leq X}$ which admit (or are equipped with) additional level structure as $X \to \infty$, and they have done so more generally over global fields. 

In recent work, many instances of this problem have been resolved. For example, Harron--Snowden \cite{Harron-Snowden} and Cullinan--Kenney--Voight \cite{Cullinan-Kenney-Voight} (see also previous work of Duke \cite{Duke} and Grant \cite{Grant})  produced asymptotics for counting those elliptic curves $E$ for which the torsion subgroup $E(\Q)_{\textup{tors}}$ of the Mordell--Weil group is isomorphic to a given finite abelian group $T$, i.e., they estimated $\#\set{E \in \scrE_{\leq X} : E(\bbQ)_{\tors} \simeq T}$ as $X \to \infty$ for each of the fifteen groups $T$ indicated in Mazur's theorem on torsion. These cases correspond to genus zero modular curves with infinitely many rational points. For such $T$, they established an asymptotic with an effectively computable constant and a power-saving error term. Moreover, satisfactory interpretations of the exponent of $X$ and the constants appearing in these asymptotics are provided. The main ingredients in the proof are the Principle of Lipschitz (also called Davenport's Lemma \cite{Davenport}) and an elementary sieve. 

Moving on, we consider asymptotics for
\begin{equation}
\# \set{E \in \scrE_{\leq X} : \textup{$E$ admits a cyclic $N$-isogeny}} 
\label{Equation: counting N-isogenies}
\end{equation}
(where we mean that the $N$-isogeny is defined over $\Q$). Our attention is again first drawn to the cases where the modular curve $Y_0(N)$, parametrizing elliptic curves with a cyclic $N$-isogeny, has genus zero: namely, $N=1, \dots, 10, 12, 13, 16, 18, 25$. For $N \leq 4$, we again have an explicit power-saving asymptotic, with the case $N=3$ due to Pizzo--Pomerance--Voight \cite{Pizzo-Pomerance-Voight} and the case $N=4$ due to Pomerance--Schaefer \cite{Pomerance-Schaefer}. For all but four of the remaining values, namely $N=7,10,13,25$, Boggess--Sankar \cite{Boggess-Sankar} provide
at least the correct growth rate.
For both torsion and isogenies, work of Bruin--Najman \cite{Bruin-Najman} and Phillips \cite{Phillips1} extend these counts to a general number field $K$.

However, the remaining four cases have quite stubbornly resisted these methods. The obstacle can be seen in quite elementary terms.  Although there is no universal elliptic curve with a cyclic $N$-isogeny, every such elliptic curve is of the form $dy^2 = x^3 + f(t)x + g(t)$ with $f(t),g(t) \in \Q[t]$ (for $t \in \Q$ away from a finite set and $d \in \Z$ a squarefree twisting parameter).  For these four values of $N$, we have $\gcd(f(t),g(t)) \neq 1$.  Phrased geometrically, the elliptic surface over $\PP^1$ defined by $y^2=x^3+f(t)x+g(t)$ has places of additive reduction (more precisely, type II). Either way, this breaks the sieve---and new techniques are required. 

\subsection{Results}

For $X \geq 1$, let
	\begin{equation}
	\NQ(X) \colonequals \# \set{E \in \scrE_{\leq X} : \textup{$E$ admits a (cyclic)} \ 7\text{-isogeny}}.
	\end{equation}
Our main result is as follows (\Cref{Theorem: asymptotic for N(X)}).

\begin{theorem}\label{Intro Theorem: asymptotic for N(X)}
	There exist effectively computable $c_1,c_2 \in \R_{>0}$ such that for every $\epsilon > 0$, we have
	\[
	\NQ(X) = c_1 X^{1/6} \log X + c_2 X^{1/6} + O(X^{3/20 + \epsilon})
	\]
	as $X \to \infty$, where the implied constant depends on $\epsilon$. 
\end{theorem}

The constants $c_1,c_2$ in \Cref{Intro Theorem: asymptotic for N(X)} are explicitly given, and estimated numerically in \cref{Section: Computations} as $c_1 = 0.09285536\ldots$ and $c_2 \approx -0.16405$. As $7$ is prime, every 7-isogeny is cyclic, so we omit this adjective for the remainder of our paper. It turns out that no elliptic curve over $\Q$ admits two $7$-isogenies with distinct kernels (\Cref{prop:noe7isogn}), so $\NQ(X)$ also counts elliptic curves \emph{equipped with} a 7-isogeny.

The first step in our strategy to prove \Cref{Intro Theorem: asymptotic for N(X)} diverges from the methods of Boggess--Sankar \cite{Boggess-Sankar} and Phillips \cite{Phillips1}, where the twists are resolved by use of a certain modular curve (denoted by $X_{1/2}(N)$). Instead, we first count twist classes directly, as follows. Let $\Qalg$ be an algebraic closure of $\Q$. Up to isomorphism \emph{over $\Qalg$}, every elliptic curve $E$ over $\Q$ with $j(E) \neq 0,1728$ has a unique Weierstrass model \eqref{Equation: Weierstrass equation} with the additional property that $B>0$ and no prime $\ell$ has $\ell^2 \mid A$ and $\ell^3 \mid B$; such a model is called \defi{twist minimal}. (See \cref{Subsection: Twist minimal Weierstrass equations} for $j(E)=0,1728$.) Let $\twistE \subset \scrE$ be the set of twist minimal elliptic curves, and let $\twistEX \colonequals \twistE \cap \scrE_{\leq X}$ be those with height at most $X$.
Accordingly, we obtain asymptotics for
\begin{equation}
	\twistN(X) \colonequals \# \{E \in \twistEX : \textup{$E$ admits a $7$-isogeny}\}
\end{equation}
as follows (\Cref{Theorem: asymptotic for twN(X)}).

\begin{theorem}\label{Intro Theorem: asymptotic for twN(X)}
	We have 
	\[
	\twistN(X) = 3\zeta(2)c_1 X^{1/6} + O(X^{2/15} \log^{17/5} X)
	\]
	as $X \to \infty$, with $c_1$ as in \textup{\Cref{Intro Theorem: asymptotic for N(X)}}.
\end{theorem}

For an outline of the proof, see \cref{sec:decomp}. The use of the Principle of Lipschitz remains fundamental, but the sieving is more involved: we decompose the function into progressively simpler pieces that can be estimated.  (See \Cref{rmk:moduli} for a stacky interpretation.)  We then deduce \Cref{Intro Theorem: asymptotic for N(X)} from \Cref{Intro Theorem: asymptotic for twN(X)} by counting twists using a Tauberian theorem (attributed to Landau). The techniques of this paper can be adapted to handle the cases $N = 10, 13, 25$, which have places of type III additive reduction; these will be treated in upcoming work.

\subsection{Contents}

In \cref{Section: Elliptic curves and lattices}, we set up basic notation and investigate minimal twists. In \cref{Section: Some analytic trivia}, we tersely review some needed facts from analytic number theory. In \cref{Section: Estimating twN(X)}, we pull together material from the earlier sections to prove \Cref{Intro Theorem: asymptotic for twN(X)}. In \cref{Section: Working over the rationals}, we use Landau's Tauberian theorem and \Cref{Intro Theorem: asymptotic for twN(X)} to obtain \Cref{Intro Theorem: asymptotic for N(X)}. In \cref{Section: Computations}, we describe algorithms to compute the various quantities we study in this paper, and report on their outputs.

\subsection{Data availability statement}

All data generated or analyzed during this study are available upon request. We have no conflicts of interest to disclose.

\subsection{Acknowledgements}

The authors would like to thank Eran Assaf, Jesse Elliott, Mits Kobayashi, David Lowry-Duda, Robert Lemke Oliver, Taylor Petty, Tristan Phillips, Carl Pomerance, and Rakvi for their helpful comments. The authors were supported by a Simons Collaboration grant (550029, to JV).

\section{Elliptic curves and isogenies}\label{Section: Elliptic curves and lattices}

In this section, we set up what we need from the theory of elliptic curves.

\subsection{Height, minimality, and defect} \label{Subsection: Twist minimal Weierstrass equations}

We begin with some notation and terminology (repeating and elaborating upon the introduction); we refer to Silverman \cite[Chapter III]{Silverman} for background.

Let $E$ be an elliptic curve over $\Q$. Recall that a \defi{(simplified) integral Weierstrass equation} for $E$ is an affine model of the form
\begin{equation} \label{eqn:yaxb}
y^2 = x^3 + Ax + B
\end{equation}
with $A,B \in \Z$. Let 
\begin{equation} \label{eqn:HAB} 
H(A,B) \colonequals \max(\abs{4A^3},\abs{27B^2}).
\end{equation}
The largest $d \in \Z_{>0}$ such that $d^4 \mid A$ and $d^6 \mid B$ is called the \defi{minimality defect} $\md(A,B)$ of the model. We then define the \defi{height} of $E$ to be
\begin{equation}
\hht(E)=\hht(A,B) \colonequals \frac{H(A,B)}{\md(A,B)^{12}}, \label{eqn:justheight}
\end{equation}
well-defined up to isomorphism. In fact, $E$ (up to isomorphism over $\Q$) has unique \defi{minimal} model 
\[ y^2=x^3+(A/d^4)x+(B/d^6)\] with minimality defect $d=1$. Let $\scrE$ be the set of elliptic curves over $\Q$ in their minimal model, and let 
\begin{equation}
\scrE_{\leq X} \colonequals \{E \in \scrE : \hht(E) \leq X\}.
\end{equation}

Let $\Qalg$ be an algebraic closure of $\Q$. We may similarly consider all integral Weierstrass equations for $E$ which define a curve isomorphic to $E$ \emph{over $\Qalg$}---these are the \defi{twists} of $E$ (defined over $\Q$). Let $E$ have $j(E) \neq 0,1728$. We call the largest $e\in \Z_{>0}$ such that $e^2 \mid A$ and $e^3 \mid B$ the \defi{twist minimality defect} of a model \eqref{eqn:yaxb}, denoted $\tmd(A,B)$. Explicitly, we have
\begin{equation}
\tmd(E) = \tmd(A,B) \colonequals \prod_{\ell} \ell^{v_\ell}, \quad \text{where} \ v_\ell \colonequals \lfloor \min(\ord_\ell(A)/2,\ord_\ell(B)/3) \rfloor,\label{Equation: defect powers}
\end{equation}
with the product over all primes $\ell$. As above, we then define the \defi{twist height} of $E$ to be
\begin{equation} \label{eqn:htdef}
\twht(E)=\twht(A,B) \colonequals \frac{H(A,B)}{\tmd(A,B)^{6}},
\end{equation}
well-defined on the $\Qalg$-isomorphism class of $E$; and $E$ has a unique model over $\Q$ up to isomorphism over $\Qalg$ with twist minimality defect $\tmd(E) = e = 1$ and $B > 0$, which we call \defi{twist minimal}, namely,
\begin{equation} 
y^2 = x^3 + (A/e^2)x + |B|/e^3. \label{Equation: twist-reduction of an arbitrary elliptic curve}
\end{equation}

For $j=0,1728$, we choose twist minimal models as follows:
\begin{itemize}
\item If $j(E)=0$ (equivalently, $A=0$), then we take $y^2=x^3 + 1$ of twist height $27$.
\item If $j(E)=1728$ (equivalently, $B=0$), then we take $y^2=x^3+x$ of twist height $4$. 
\end{itemize}
Let $\twistE \subset \scrE$ be the set of twist minimal elliptic curves, and let $\twistEX \colonequals \twistE \cap \scrE_{\leq X}$ be those with twist height at most $X$.  If $E \in \scrE$ has $j(E) \neq 0,1728$, then the set of twists of $E$ in $\scrE$ are precisely those of the form $E^{(c)} \colon y^2 = x^3 + c^2 A x + c^3 B$ for $c \in \Z$ squarefree, and 
\begin{equation}
\hht(E^{(c)})=c^6 \twistheight(E).\label{Equation: quadratic twists multiply height by c^6}
\end{equation} 
If further $E \in \twistE$, then of course $\twistheight(E)=\hht(E)$.  (For $j(E)=0,1728$, we instead have sextic and quartic twists, but these will not figure here: see \Cref{prop:noe7isogn}.)  

\begin{remark} \label{rmk:moduli0}
This setup records in a direct manner the more intrinsic notions of height coming from moduli stacks. The moduli stack $Y(1)_\Q$ of elliptic curves admits an open immersion into a weighted projective line $Y(1) \hookrightarrow \PP(4,6)_\Q$ by $E \mapsto (A:B)$ for any choice of model \eqref{eqn:yaxb}, and the height of $E$ is the height of the point $(A:B) \in \PP(4,6)(\Q)$ associated to $\scrO_{\PP(4,6)}(12)$ (with coordinates harmlessly scaled by $4,27$): see Bruin--Najman \cite[\S 2, \S 7]{Bruin-Najman} and Phillips \cite[\S 2.2]{Phillips1}. Similarly, the height of the twist minimal model is given by the height of the point $(A:B) \in \PP(2,3)(\Q)$ associated to $\scrO_{\PP(2,3)}(6)$, which is almost but not quite the height of the $j$-invariant (in the usual sense).  
\end{remark}

\subsection{Isogenies of degree $7$} \label{sec:7isog}

Next, we gather the necessary input from modular curves. Recall that the modular curve $Y_0(7)$, defined over $\Q$, parametrizes pairs $(E,\phi)$ of elliptic curves $E$ equipped with a $7$-isogeny $\phi$ up to isomorphism, or equivalently, a cyclic subgroup of order $7$ stable under the absolute Galois group $\Gal_\Q \colonequals \Gal(\Qalg\,|\,\Q)$.  For further reference on the basic facts on modular curves used in this section, see e.g.\ Diamond--Shurman \cite{Diamond-Shurman} and Rouse--Sutherland--Zureick-Brown \cite[\S 2]{RSZB}. We compute that the coarse space of $Y_0(7)$ is an affine open in $\PP^1$, so the objects of interest are parametrized by its coordinate $t \neq -7,\infty$ (see \Cref{lem:7isog-param}). 

More precisely, define 
\begin{equation} \label{Equation: Definition of f for 7-isogenies}
\begin{aligned}
	f_0(t) &\colonequals -3 (t^2 - 231 t + 735) \\
	&= -3 (t^2 - (3 \cdot 7 \cdot 11)t + (3 \cdot 5 \cdot 7^2)), \\
	g_0(t) &\colonequals 2 (t^4 + 518 t^3 - 11025 t^2 + 6174 t - 64827) \\
	&= 2 (t^4 + (2 \cdot 7 \cdot 31)t^3 - (3^2 \cdot 5^2 \cdot 7^2)t^2 + (2 \cdot 3^2 \cdot 7^3)t - (3^3 \cdot 7^4)), \\
	h(t) &\colonequals t^2 + t + 7, \\
	f(t) &\colonequals f_0(t) h(t), \\
	g(t) &\colonequals g_0(t) h(t).
\end{aligned}
\end{equation}
Then $h(t) = \gcd(f(t), g(t))$.

\begin{lem} \label{lem:7isog-param}
The set of elliptic curves $E$ over $\Q$ that admit a $7$-isogeny (defined over $\Q$) are precisely those of the form $E \colon y^2 = x^3 + c^2f(t)x + c^3g(t)$ for some $c \in \Q^\times$ and $t \in \Q$ with $t \neq -7$. 
\end{lem}

\begin{proof}
Routine calculations with $q$-expansions for modular forms on the group $\Gamma_0(7)$, with the cusps at $t=-7,\infty$ show that every elliptic curve $E$ over $\bbQ$ that admits a 7-isogeny is a twist of
\[
E : y^2 = x^3 + f(t) x + g(t)
\]
for some $t \in \bbQ$. But $f(t)$ and $g(t)$ have no roots in $\bbQ$, so these twists must be quadratic, as desired. See \cite[Proposition 3.3.16]{Cullinan-Kenney-Voight} for a similar but more expansive argument.
\end{proof}

Of course, for elliptic curves up to isomorphism over $\Qalg$, we can ignore the factor $c$ in \Cref{lem:7isog-param}.

\begin{remark}
	Let
	\begin{equation} \label{Equation: Definition of f for 7-isogenies, isogenous}
\begin{aligned}
	f\prm_0(t) &\colonequals -3 (t^2 + 9 t + 15) \\
	&= -3 (t^2 + (3^2) t + (3 \cdot 5)), \\
	g\prm_0(t) &\colonequals 2 (t^4 + 14 t^3 + 63 t^2 + 126 t + 189) \\
	&= 2 (t^4 + (2 \cdot 7) t^3 + (3^2 \cdot 7) t^2 + (2 \cdot 3^2 \cdot 7) t + (3^3 \cdot 7)), \\
	f\prm(t) &\colonequals f\prm_0(t) h(t), \\
	g\prm(t) &\colonequals g\prm_0(t) h(t),
\end{aligned}
\end{equation}
	with $h(t)$ as above. The elliptic curve $E$ in \Cref{lem:7isog-param} is $7$-isogenous to 
	\[
	E\prm : y^2 = x^3 + c^2 f\prm(t) + c^3 g\prm(t)
	\]
	via the marked $7$-isogeny. Naturally, $E\prm$ is also isogeneous to $E$ via the dual $7$-isogeny. We obtain \eqref{Equation: Definition of f for 7-isogenies, isogenous} from \eqref{Equation: Definition of f for 7-isogenies} via the Atkin-Lehner involution, which in our coordinates is given by 
	\begin{equation}
	w_7 : t \mapsto -\frac{7t}{t+7}.
	\end{equation}
	All of our arguments below could have been applied equally well using the parameterization \eqref{Equation: Definition of f for 7-isogenies, isogenous} instead of the parameterization \eqref{Equation: Definition of f for 7-isogenies}.
\end{remark}

\begin{prop} \label{prop:noe7isogn}
No elliptic curve $E$ over $\Q$ admits two $7$-isogenies with distinct kernels, and no
$E$ over $\Q$ with $j(E)=0,1728$ admits a $7$-isogeny.
\end{prop}

\begin{proof}
For the first statement: if $E$ admits two distinct $7$-isogenies, then generators for each kernel give a basis for the $7$-torsion of $E$ in which $\Gal_\Q$ acts diagonally. The corresponding compactified modular curve, $X_{\textup{sp}}(7)$, has genus $1$ and $2$ rational cusps; it is isomorphic to $X_0(49)$ over $\Q$, and has Weierstrass equation $y^2+xy=x^3-x^2-2x-1$ and LMFDB label \href{https://lmfdb.org/EllipticCurve/Q/49/a/4}{\textsf{49.a4}}. Its Mordell--Weil group is $\Z/2\Z$, so all rational points are cusps.

For the second statement, we simply observe that $f(t)$ and $g(t)$ have no roots $t \in \Q$. 
\end{proof}

To work with integral models, we take $t=a/b$ (in lowest terms) and homogenize, giving the following polynomials in $\Z[a,b]$:
\begin{equation} \label{eqn:ABCab}
\begin{aligned}
	C(a, b) &\colonequals b^2 h(a/b)=a^2 + ab + 7 b^2, \\
	\tA(a, b) &\colonequals b^2 f_0(a/b) = -3 (a^2 - 231 a b + 735 b^2), \\
	\tB(a, b) &\colonequals b^4 g_0(a/b) = 2 (a^4 + 518 a^3 b - 11025 a^2 b^2 + 6174 a b^3 - 64827 b^4), \\
A(a,b) &\colonequals b^4 f(a/b) = C(a,b)\tA(a,b) \\
B(a,b) &\colonequals b^6 f(a/b) = C(a,b)\tB(a,b). \\ 
\end{aligned}
\end{equation}
We have $C(a,b) = \gcd(A(a,b),B(a,b)) \in \Z[a,b]$.

We say that a pair $(a, b) \in \bbZ^2$ is \defi{groomed} if $\gcd(a, b) = 1$, $b > 0$, and $(a, b) \neq (-7, 1)$. Thus \Cref{lem:7isog-param} and \Cref{prop:noe7isogn} provide that the elliptic curves $E \in \scrE$ that admit a $7$-isogeny are precisely those with a model 
\begin{equation}
y^2 = x^3 + \frac{c^2 A(a, b)}{d^4} x + \frac{c^3 B(a, b)}{d^6} \label{Equation: Weierstrass equation for elliptic surface, homogenized}
\end{equation}
where $(a, b)$ is groomed, $c \in \Z$ is squarefree, and $d=\md(c^2A(a,b),c^3B(a,b))$. 
 Thus the count
\begin{equation} \label{eqn:NQx}
\NQ(X) \colonequals \#\{E \in \scrE_{\leq X} : \textup{$E$ admits a $7$-isogeny}\} 
\end{equation}
can be computed as
\begin{equation} \label{eqn:NQX}
\NQ(X) = \#\left\{(a, b, c) \in \bbZ^3 : 
\begin{minipage}{33ex} $(a, b)$ groomed, $c$ squarefree, and \\
$\hht(c^2 A(a,b),c^3 B(a,b)) \leq X$
\end{minipage}
\right\}.
\end{equation}
with the height defined as in \eqref{eqn:justheight}.

Similarly, but more simply, the subset of $E \in \twistE$ that admit a $7$-isogeny are 
\begin{equation}
E_{a,b} \colon y^2 = x^3 + \frac{A(a, b)}{e^2} x + \frac{|B(a, b)|}{e^3} \label{Equation: Weierstrass equation for elliptic surface, homogenized, twist}
\end{equation}
with $(a, b)$ groomed and $e=\twistdefect(A(a,b),B(a,b))$ the twist minimality defect \eqref{Equation: defect powers}. Accordingly, if we define
\begin{equation} \label{eqn:NQxwist}
\twistN(X) \colonequals \# \{E \in \twistEX : \textup{$E$ admits a $7$-isogeny}\}
\end{equation}
then
\begin{equation} \label{eqn:twistheightcalc}
\twistN(X) = \# \set{(a, b) \in \bbZ^2 : \textup{$(a, b)$ groomed and $\twht(A(a,b),B(a,b)) \leq X$}}.
\end{equation}

\begin{remark} \label{rmk:moduli}
Returning to \Cref{rmk:moduli0}, we conclude that counting elliptic curves equipped with a $7$-isogeny is the same as counting points on $\PP(4,6)_\Q$ in the image of the natural map $Y_0(7) \to Y(1) \subseteq \PP(4,6)_\Q$.  Counting them up to twist replaces this with the further natural quotient by $(a : b) \sim (\lambda^2 a : \lambda^3 b)$ for $(a : b) \in \PP(4, 6)_\Q$ and $\lambda \in \bbQ^\times$, which gives us $\PP(2,3)_\Q$.  
\end{remark}

\subsection{Twist minimality defect}\label{Subsection: groomed and ungroomed primes}

The twist minimality defect is the main subtlety in our study of $\twistN(X)$, so we analyze it right away. 

\begin{lemma}\label{Lemma: 3 and 7 are the ungroomed primes}
	Let $(a,b) \in \Z^2$ be groomed, let $\ell$ be prime, and let $v \in \Z_{\geq 0}$. Then the following statements hold.
	\begin{enumalph}
	\item If $\ell \neq 3, 7$, then $\ell^v \mid \twistdefect(A(a,b),B(a,b))$ if and only if $\ell^{3v} \mid C(a, b)$.
	\item $\ell^{3v} \mid C(a,b)$ if and only if $\ell \nmid b$ and $h(a/b) \equiv 0 \pmod{\ell^{3v}}$.
	\item If $\ell \neq 3$, then $\ell \mid C(a,b)$ implies $\ell \nmid (2a+b)=(\partial C/\partial a)(a,b)$.
	\end{enumalph}
\end{lemma}

\begin{proof}
We use the notation \eqref{eqn:ABCab} and argue as in Cullinan--Kenney--Voight \cite[Proof of Theorem 3.3.1, Step 3]{Cullinan-Kenney-Voight}. 
	For part (a), we compute the resultants
\[
\Res(A_0(t,1),B_0(t,1))=\Res(f_0(t),g_0(t))=-2^8\cdot 3^7 \cdot 7^{14} = \Res(A_0(1,u),B_0(1,u)).
\]
So if $\ell \neq 2,3,7$, then $\ell \nmid \gcd(A_0(a,b),B_0(a,b))$; so by \eqref{Equation: defect powers}, if $\ell^v \mid \tmd(A(a,b),B(a,b))$ then $\ell^{2v} \mid C(a,b)$. But also
\[
\Res(B_0(t,1),C(t,1))=\Res(g_0(t),h(t)) = 2^8 \cdot 3^3 \cdot 7^7 = \Res(B_0(1,u),C(1,u)),
\]
so $\ell \nmid \gcd(B_0(a,b), C(a, b))$ and thus $\ell^v \mid \tmd(A(a,b),B(a,b))$ if and only if $\ell^{3v} \mid C(a,b)$. If $\ell = 2$, a short computation confirms that $B(a, b)$ is twice an odd integer whenever $(a, b)$ is groomed, so our claim also holds in this case. 

For (b), by homogeneity it suffices to show that $\ell \nmid b$, and indeed this holds since if $\ell \mid b$ then $A(a,0) \equiv -3a^4 \equiv 0 \pmod{\ell}$ and $B(b,0) \equiv 2a^6 \equiv 0 \pmod{\ell}$ so $\ell \mid a$, a contradiction.
	
Part (c) follows from (b) and the fact that $h(t)$ has discriminant $\disc(h(t))=3^3$. 
\end{proof}

For $e \geq 1$, let $\widetilde{\calT}(e) \subseteq (\Z/e^3\Z)^2$ denote the image of
\[
\set{(a, b) \in \bbZ^2 : (a, b) \ \text{groomed}, \ e \mid \twistdefect(A (a, b), B (a, b))}
\]
under the projection
\[
\bbZ^2 \to (\bbZ / e^3 \bbZ)^2,
\]
and let $\tcT(e) \colonequals \# \widetilde{\calT}(e)$. Similarly, let $\calT(e) \subseteq \Z/e^3\Z$ denote the image of
\[
\set{t \in \bbZ : e^2 \mid f(t) \ \text{and} \ e^3 \mid g(t)}
\]
under the projection
\[
\bbZ \to \bbZ / e^3 \bbZ,
\]
and let $\cT(e) \colonequals \#\calT(e)$.

As usual, we write $\varphi(n) \colonequals n \prod_{p \mid n} (1 - 1/p)$ for the Euler totient function.

\begin{lemma}\label{Lemma: bound on T(e)}
The following statements hold.
\begin{enumalph}
\item The functions $\widetilde{T}(e)$ and $T(e)$ are multiplicative, and $\widetilde{T}(e) = \varphi(e^3) T(e)$.
\item For all $\ell \neq 3,7$ and $v \geq 1$, 
	\begin{align*}
	\cT(\ell^v) &= \cT(\ell) = 1 + \left(\frac{\ell}{3}\right).
	\end{align*}
\item We have
	\[
	\cT(3) = 18, \ \cT(3^2) = 27, \text{ and }\ \cT(3^v) = 0 \ \text{for} \ v \geq 3,
	\]
	and
	\[
	\cT(7) = 50, \ \cT(7^2) = 7^4+1=2402, \text{ and }\ \cT(7^v) = 7^7+1=823544 \ \text{for} \ v \geq 3.
	\]
\item 
	We have $\cT(e) =O(2^{\omega(e)})$,	where $\omega(e)$ is the number of distinct prime divisors of $e$.
	\end{enumalph}
\end{lemma}

\begin{proof}
For part (a), multiplicativity follows from the CRT (Sun Zi theorem). For the second statement, let $\ell$ be a prime, and let $e = \ell^v$ for some $v \geq 1$. Consider the injective map
\begin{equation}
\begin{aligned}
\calT(\ell^v) \times (\bbZ/ \ell^{3 v})^\times &\to \tcalT(\ell^v) \\
(t,u) &\mapsto (tu,u)
\end{aligned}
\end{equation}
We observe $A(1, 0) = -3$ and $B(1, 0) = 2$ are coprime, so no pair $(a, b)$ with $b \equiv 0 \pmod \ell$ can be a member of $\widetilde{\calT}(\ell^v)$. Surjectivity of the given map follows, and counting both sides gives the result.
	
Now part (b). For $\ell \neq 3, 7$, \Cref{Lemma: 3 and 7 are the ungroomed primes}(a)--(b) yield
	\[
	\calT(\ell^v) = \set{t \in \bbZ/\ell^{3v} \bbZ : h(t) \equiv 0 \psmod{\ell^{3v}}}.
	\]
By \Cref{Lemma: 3 and 7 are the ungroomed primes}(c), $h(t) \equiv 0 \pmod \ell$ implies $h\prm(t) \not\equiv 0 \pmod \ell$, so Hensel's lemma applies and we need only count roots of $h(t)$ modulo $\ell$, which by quadratic reciprocity is 
\[
1 + \parent{\frac{-3}{\ell}} = 1 + \parent{\frac{\ell}{3}} = \begin{cases}
2, & \ \text{if} \ \ell \equiv 1 \psmod 3; \\
0, & \ \text{else.}
\end{cases}
\]

Next, part (c). For $\ell = 3$, we just compute 
$T(3) = 18$, $T(3^2) = 27$, and $T(3^3) = 0$; then $T(3^3) = 0$ implies $T(3^v) = 0$ for all $v \geq 3$. 
For $\ell = 7$, we compute
\[
\cT(7) = 50, \ \cT(7^2) = 2402, \ \cT(7^3) = \dots = \cT(7^6) = 823544.
\]
Hensel's lemma still applies to $h(t)$: let $t_0,t_1$ be the roots of $h(t)$ in $\Z_7$ with $t_0 \colonequals 248044 \pmod{7^7}$ (so that $t_1=-1-t_0$). We claim that
\begin{equation}
\calT(7^{3v}) = \set{t_0} \sqcup \set{t_1 + 7^{3v - 7} u \in \bbZ / 7^{3v} \bbZ :u \in \bbZ / 7^7 \bbZ},\label{Equation: decomposition of T(7^v)}
\end{equation}
for $3v \geq 7$. Indeed, $g_0(t_1) \equiv 0 \pmod {7^7}$, so we can afford to approximate $t_1$ modulo $7^{3v - 7}$. As $g(t_0) \not\equiv 0 \pmod {7}$ and $g(t_1) \not\equiv 0 \pmod{7^8}$, no other values of $t$ suffice. Thus $T(7^{3v}) = 1 + 7^7 = 823544$.

Finally, part (d). From (a)--(c) we conclude
\begin{equation}
T(e) \leq \frac{27 \cdot 823544}{4} \cdot \prod_{\substack{\ell \mid e \\ \ell \neq 3,7}} \parent{1 + \parent{\frac{\ell}{3}}} \leq 5558922 \cdot 2^{\omega(e)}
\end{equation}
so $T(e) = O(2^{\omega(e)})$ as claimed.
\end{proof}

\subsection{The common factor \texorpdfstring{$C(a, b)$}{Cab}}

In view of \Cref{Lemma: 3 and 7 are the ungroomed primes}, the twist minimality defect away from the primes $2,3,7$ is given by the quadratic form $C(a,b)=a^2+ab+7b^2=b^2 h(a/b)$. Fortunately, this is the norm form of a quadratic order of class number $1$, so although this is ultimately more than what we need, we record some consequences of this observation which take us beyond \Cref{Lemma: bound on T(e)}.

For $m \in \bbZ_{>0}$, let
\begin{equation}\label{Equation: c(m)}
	c(m) \colonequals \#\{(a, b) \in \bbZ^2 : b > 0, \ \gcd(a, b) = 1, \ C(a, b) = m\}.
\end{equation}

\begin{lemma}\label{Lemma: c(m) is multiplicative etc}
	The following statements hold.
	\begin{enumalph}	
	\item We have $c(m n) = c(m) c(n)$ for $m,n \in \Z_{>0}$ coprime.
	\item We have
	\[
	c(3) = 0, \ c(3^2) = 2, \ c(3^3) = 3, \ \text{and} \ c(3^v) = 0 \ \text{for} \ v \geq 4;
	\]
	for $p \neq 3$ prime and $k \geq 1$ an integer, we have
\begin{equation} \label{eqn:cp}
	c(p) = c(p^k) = 1 + \parent{\frac{p}{3}}.
	\end{equation}
	\item For $m$ and $n$ positive integers, we have
	\[
	c(n^3 m) \leq 3 \cdot 2^{\omega(n)-1} c(m).
	\]
	\end{enumalph}
\end{lemma}

\begin{proof}
	Let $\zeta \colonequals (1 + \sqrt{-3})/2$, 
	so $\overline{\zeta} = 1 - \zeta = (1 - \sqrt{-3})/2$. 
	The quadratic form 
	\[
	C(a, b) = a^2 + a b + 7 b^2 = (a+b\parent{-1 + 3 \zeta})(a+b\overline{\parent{-1 + 3 \zeta}}) =\Nm(a+b\parent{-1 + 3 \zeta})
	\]
	is the norm on the order $\bbZ[3 \zeta]$ in basis $\set{1, -1 + 3 \zeta}$. Recall that $\alpha \in \bbZ[3 \zeta]$ is \defi{primitive} if no $n \in \bbZ_{>1}$ divides $\alpha$. Thus, accounting for sign,
	\begin{equation}
	2c(m) = \#\{\alpha \in \bbZ[3 \zeta] \ \text{primitive} : \Nm(\alpha) = m\}.\label{Equation: c(m) as an algebraic number theory thing}
	\end{equation}
	
	The order $\bbZ[3 \zeta]$ is a suborder of the Euclidean domain $\bbZ[\zeta]$ of conductor 3. It inherits from $\bbZ[\zeta]$ the following variation on unique factorization: up to sign, every nonzero $\alpha \in \Z[3\zeta]$ can be written uniquely %(up to reordering) 
	as 
	\[
	\alpha = \beta \pi_1^{e_1} \cdots \pi_r^{e_r},
	\]
	where $\Nm(\beta)$ is a power of $3$, $\pi_1, \dots, \pi_r$ are distinct irreducibles coprime to $3$, and $e_1, \dots, e_r$ are positive integers. Note that $\alpha$ is primitive if and only if $\beta$ is primitive and for $1 \leq i, j \leq r$ (not necessarily distinct) we have $\pi_i \neq \overline{\pi_j}$. Thus if $m$ and $n$ are coprime integers, $\alpha \in \bbZ[3 \zeta]$ is primitive, and $\Nm(\alpha) = mn$, then $\alpha$ may be factored uniquely (up to sign) as $\alpha = \alpha_1 \alpha_2$, where $\Nm(\alpha_1) = m$ and $\Nm(\alpha_2) = n$. This proves (a). 

	We now prove (b). If $p \neq 3$ is inert in $\bbZ[3\zeta]$ (equivalently, in $\bbZ[\zeta]$), then no primitive $\alpha$ satisfies $\Nm(\alpha) = p^v$, so $c(p^v) = 0$. If $p \neq 3$ splits in $\bbZ[3 \zeta]$ (equivalently, in $\bbZ[\zeta]$), then no primitive $\alpha$ is divisible by more than one of the two primes above $p$, so $c(p^v) = 2$. This proves \eqref{eqn:cp} (compare \Cref{Lemma: bound on T(e)}). Finally, if $p = 3$, we compute $c(3) = 0,$ $c(3^2) = 2,$ and $c(3^3) = 3$. Congruence conditions show $c(3^v) = 0$ for $v \geq 4$.

	Part (c) follows immediately from (a) and (b).
\end{proof}

\begin{remark}
We prove \Cref{Lemma: c(m) is multiplicative etc}(a) and \Cref{Lemma: c(m) is multiplicative etc}(b) only as a means to proving \Cref{Lemma: c(m) is multiplicative etc}(c). Although the algebraic structure of the Eisenstein integers $\bbZ[\zeta]$ may not be available in the study of other families of elliptic curves that exhibit potential additive reduction, we expect analogues of \Cref{Lemma: c(m) is multiplicative etc}(c) to hold in a general context.
\end{remark}

The twist minimality defect measures the discrepancy between $H(A, B)$ and $\twistheight(A, B)$: this discrepancy cannot be too large compared to $C(a,b)$, as the following theorem shows.

\begin{theorem}\label{Theorem: Controlling size of twist minimality defect}
	We have the following.
	\begin{enumalph}
	\item For all $(a, b) \in \bbR^2$, we have
	\begin{equation}
	108 C(a, b)^6 \leq \rawheight(A(a, b), B(a, b)) \leq \upperratio C(a, b)^6,\label{Equation: upper and lower bounds for H}
	\end{equation}
	where $\upperratio = 311\,406\,871.990\,204\ldots$ is an explicit algebraic number.
	\item If $C(a, b) = e_0^3 m$, with $m$ cubefree, then $\twistdefect(A(a, b), B(a, b)) = e_0 e\prm$ for some $e\prm \mid 3 \cdot 7^3$, and
	\[
	\frac{2^2}{3^3 \cdot 7^{18}} e_0^{12} m^6 \leq \twistheight(A(a, b), B(a, b)) \leq \upperratio e_0^{12} m^6.
	\]
	\end{enumalph}
\end{theorem}

\begin{proof}
	We wish to find the extrema of $\rawheight(A(a, b), B(a, b))/C(a, b)^6$. As this expression is homogeneous of degree 0, and $C(a, b)$ is positive definite, we may assume without loss of generality that $C(a, b) = 1$.  Using Lagrange multipliers, % and examining the critical points of $\rawheight(A(a, b), B(a, b))$ subject to $C(a, b) = 1$, 
	we verify that \eqref{Equation: upper and lower bounds for H} holds: the lower bound is attained at $(1, 0)$, and the upper bound is attained when $a = 0.450\,760\dots$ and $b=-0.371\,118\dots$ are roots of
	\begin{equation}\label{Equation: Roots defining upperratio}
	\begin{aligned}
		 1296 a^8 - 2016 a^6 + 2107 a^4 - 1596 a^2 + 252 &=0  \\
%		 =& 2^4 \cdot 3^4 \cdot a^8 - 2^5 \cdot 3^2 \cdot 7 \cdot a^6 + 7^2 \cdot 43 \cdot a^4 - 2^2 \cdot 3 \cdot 7 \cdot 19 \cdot a^2 + 2^2 \cdot 3^2 \cdot 7, \ \text{and} \\
		 1067311728 b^8 - 275298660 b^6 + 43883077 b^4 - 3623648 b^2 + 1849 &= 0, \\
%		 =& 2^4 \cdot 3^4 \cdot 7^7 \cdot b^8 - 2^2 \cdot 3^2 \cdot 5 \cdot 7^6 \cdot 13 \cdot b^6 + 7^6 \cdot 373 \cdot b^4 - 2^5 \cdot 7^2 \cdot 2311 \cdot b^2 + 43^2
	\end{aligned}
	\end{equation}
respectively. %(4.980802388898044, 0.45076099660469304, -0.37111801138274486, 311406871.99020463)
	
	Now write $C(a, b) = e_0^3 m$ with $m$ cubefree, and write $\twistdefect(A(a, b), B(a, b)) = e_0 e\prm$. By \Cref{Lemma: 3 and 7 are the ungroomed primes}, $e\prm = 3^v 7^w$ for some $v, w \geq 0$; a short computation shows $v \in \set{0, 1}$, and \eqref{Equation: decomposition of T(7^v)} shows $w \leq \ceil{7/3} = 3$. As 
	\[
	\rawheight(A(a, b), B(a, b)) = e_0^6 \parent{e\prm}^6 \twistheight(A(a, b), B(a, b)),
	\]
	we see
	\[
	\frac{108}{(e\prm)^6} e_0^{12} m^6 \leq \twistheight(A(a, b), B(a, b)) < \frac{\upperratio}{(e\prm)^6} e_0^{12} m^6.
	\]
	Rounding $e\prm$ up to $3 \cdot 7^3$ on the left and down to $1$ on the right gives the desired result.
\end{proof}

\begin{corollary}\label{Corollary: bound twist defect in terms of twist height}
	Let $(a, b)$ be a groomed pair. We have
\[
	\twistdefect(A(a, b), B(a, b)) \leq \frac{3^{5/4} \cdot 7^{9/2}}{2^{1/6}} \twistheight(A(a, b), B(a, b))^{1/12} \]
where $3^{5/4} \cdot 7^{9/2} / 2^{1/6} = 22\,344.5\ldots$ 
\end{corollary}

\begin{proof}
	In the notation of \Cref{Theorem: Controlling size of twist minimality defect}(b),
	\[
	e_0^{12} m^6 \leq \frac{3^3 \cdot 7^{18}}{2^2} \twistheight(A(a, b), B(a, b)).
	\]
	Multiplying through by $(e\prm)^{12}$, rounding $m$ down to $1$ on the left, rounding $e\prm$ up to $3 \cdot 7^7$ on the right, and taking $12$th roots of both sides, we obtain the desired result.
\end{proof}

\section{Analytic ingredients}\label{Section: Some analytic trivia}

In this section, we record some results from analytic number theory used later.

\subsection{Lattices and the principle of Lipschitz}\label{Subsection: Lattices and the principle of Lipschitz}

We recall (a special case of) the Principle of Lipschitz, also known as Davenport's Lemma.

\begin{theorem}[Principle of Lipschitz]\label{Theorem: Principle of Lipschitz}
	Let $\calR \subseteq \bbR^2$ be a closed and bounded region, with rectifiable boundary $\partial \calR$. We have
	\[
	\#(\calR \cap \bbZ^2) = \Area(\calR) + O(\len(\partial \calR)),
	\]
	where the implicit constant depends on the similarity class of $\calR$, but not on its size, orientation, or position in the plane $\bbR^2$.
\end{theorem}

\begin{proof}
	See Davenport \cite{Davenport}.
\end{proof}

Specializing to the case of interest, for $X > 0$ let
\begin{equation} \label{eqn: R(X)}
\calR(X) \colonequals \set{(a, b) \in \bbR^2 :\rawheight(A(a, b), B(a, b)) \leq X, \ b \geq 0},
\end{equation}
and let $R \colonequals \Area(\calR(1))$. The region $\calR(1)$ is the common region in Figure \ref{Figure: calR(1)}.

\jvfigure{Figure: calR(1)}{\includegraphics[scale=0.3]{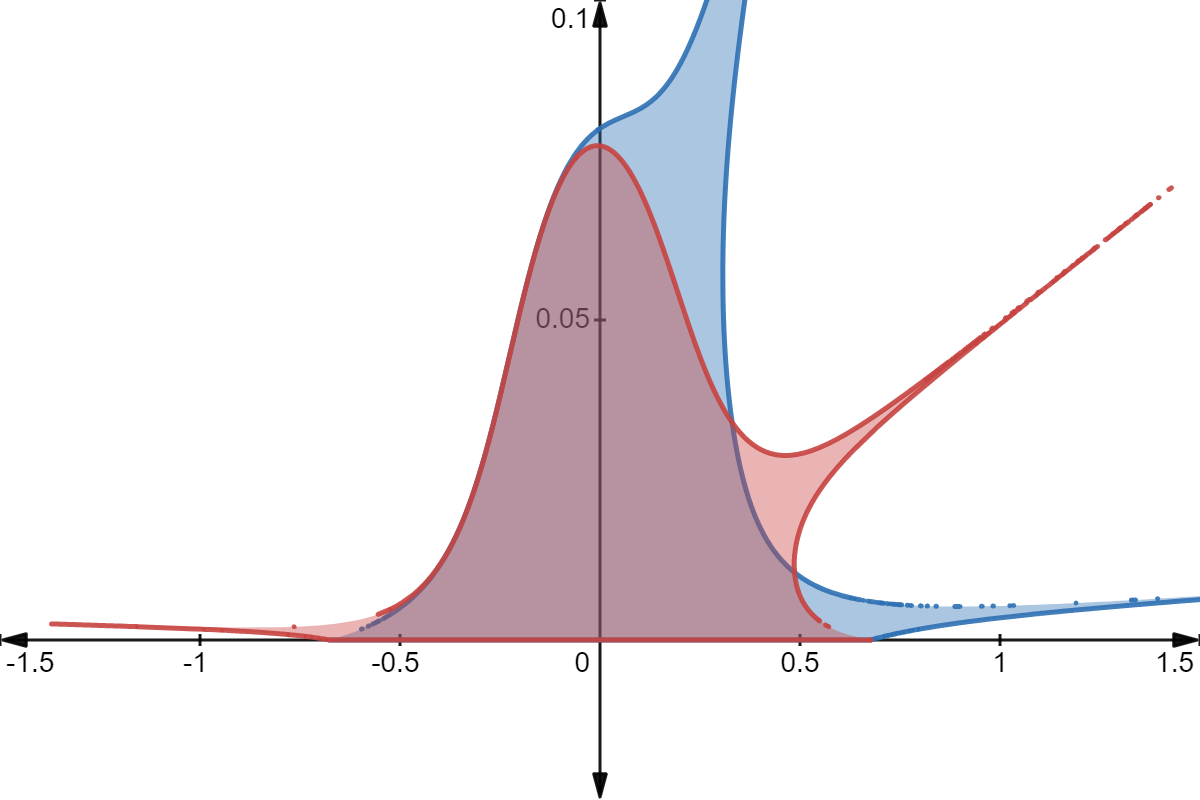}}{The region $\calR(1)$}

\begin{lemma}\label{Lemma: formula for R(X)}
	For $X > 0$, we have $\Area (\calR(X)) = R X^{1/6}$.
\end{lemma}

\begin{proof}
	Since $f(t)=A(t,1)$ and $g(t)=B(t,1)$ have no common real root, the region $\calR(X)$ is compact \cite[Proof of Theorem 3.3.1, Step 2]{Cullinan-Kenney-Voight}. The homogeneity
	\[
 	\rawheight(A(u a, ub), B(ua, ub)) = u^{12} \rawheight(A(a, b), B(a, b))
	\]
	implies
	\[
	\Area(\calR(X)) = \Area(\{(X^{1/12} a, X^{1/12} b) :(a, b) \in \calR(1)\}) = X^{1/6} \Area(\calR(1)) = R X^{1/6}
	\]
	as desired.
\end{proof}

The following corollaries are immediate.

\begin{corollary}\label{Corollary: Estimates for L(X)}
For $a_0, b_0, d \in \bbZ$ with $d \geq 1$, we have
	\[
	\#\{(a, b) \in \calR(X) \cap \bbZ^2 : (a, b) \equiv (a_0, b_0) \psmod d\} = \frac{R X^{1/6}}{d^2} + O\parent{\frac{X^{1/{12}}}{d}}.
	\]
	The implied constants are independent of $X$, $d$, $a_0,$ and $b_0$. In particular,
	\begin{equation}
	\#(\calR(X) \cap \bbZ^2) = R X^{1/6} + O(X^{1/{12}}).
	\end{equation}
\end{corollary}

\begin{proof}
Combine \Cref{Lemma: formula for R(X)} and \Cref{Theorem: Principle of Lipschitz}.
\end{proof}

\begin{corollary}\label{Corollary: Bound on partial sums of c(m)}
	Let $\parent{c(m)}_{m \geq 1}$ be as in \eqref{Equation: c(m)}. We have
	\[
	\sum_{m \leq X} c(m) = O(X).
	\]
\end{corollary}

\begin{proof}
Immediate from \Cref{Corollary: Estimates for L(X)}.
\end{proof}

\subsection{Dirichlet series}

The following theorem is attributed to Stieltjes.

\begin{theorem}\label{Theorem: product of Dirichlet series converges}
	Let $\alpha, \beta : \bbZ_{> 0} \to \bbR$ be arithmetic functions. If $L_\alpha(s) \colonequals \sum_{n \geq 1} \alpha(n) n^{-s}$ and $L_\beta(s) \colonequals \sum_{n \geq 1} \beta(n) n^{-s}$ both converge when $\repart(s) > \sigma$, and one of these two series converges absolutely, then
	\[
	L_{\alpha \ast \beta}(s) \colonequals \sum_{n \geq 1} \parent{\sum_{d \mid n} \alpha(d) \beta\parent{\frac nd}} n^{-s}
	\]
	converges for $s$ with $\repart(s) > \sigma$. If both $L_\alpha(s)$ and $L_\beta(s)$ both converge absolutely when $\repart(s) > \sigma$, then so does $L_{\alpha \ast \beta}(s)$.
\end{theorem}

\begin{proof}
	Widder \cite[Theorems 11.5 and 11.6b]{Widder} proves a more general result, or see Tenenbaum \cite[proof of Theorem II.1.2, Notes on p.~204]{Tenenbaum}.
\end{proof}

Let $\gamma \colonequals \lim_{y \to \infty} \bigl(\sum_{n \leq y} 1/n\bigr) - \log y$ be the Euler--Mascheroni constant. 

\begin{theorem}\label{Theorem: Laurent series expansion for zeta(s)}
	The difference
	\[
	\zeta(s) - \parent{\frac{1}{s - 1} + \gamma}
	\]
	is entire on $\bbC$ and vanishes at $s = 1$. 
\end{theorem}

\begin{proof}
	Ivi\'c \cite[page 4]{Ivic} proves a more general result.
\end{proof}

\subsection{Regularly varying functions}

We require a fragment of Karamata's integral theorem for regularly varying functions.

\begin{definition}
	Let $F \colon \bbR_{\geq 0} \to \bbR$ be measurable and eventually positive. We say that $F$ is \defi{regularly varying of index $\rho \in \bbR$} if for each $\lambda > 0$ we have
	\[
	\lim_{y \to \infty} \frac{F(\lambda y)}{F(y)} = \lambda^\rho.
	\]
\end{definition}

\begin{theorem}[Karamata's integral theorem]\label{Theorem: Karamata's integral theorem}
	Let $F \colon \bbR_{\geq 0} \to \bbR$ be locally bounded and regularly varying of index $\rho \in \bbR$. Let $\sigma \in \bbR$. Then the following statements hold.
	\begin{enumalph}
		\item For any $\sigma > \rho + 1$, we have
		\[
		\int_y^\infty t^{-\sigma} F(u) \,\mathrm{d}u \sim \frac{y^{1 - \sigma} F(y)}{\abs{\sigma - \rho - 1}}
		\]
		as $y \to \infty$.
		\item For any $\sigma < \rho + 1$, we have
		\[
		\int_0^y u^{-\sigma} F(u) \,\mathrm{d}u \sim \frac{y^{1 - \sigma} F(y)}{\abs{\sigma - \rho - 1}}
		\]
	as $y \to \infty$.
	\end{enumalph}
\end{theorem}

\begin{proof}
	See Bingham--Glodie--Teugels \cite[Theorem 1.5.11]{Bingham-Goldie-Teugels}. (Karamata's integral theorem also includes a converse.)
\end{proof}

\begin{corollary}\label{Corollary: tail of sum of f(n)/n^sigma}
	Let $\alpha \colon \bbZ_{> 0} \to \bbR$ be an arithmetic function, and suppose that for some $\kappa, \rho, \tau \in \bbR$ with $\kappa \neq 0$ and $\rho > 0$, we have
	\begin{equation}
	F(y) \colonequals \sum_{n \leq y} \alpha(n) \sim \kappa y^\rho \log^\tau y\label{Equation: asymptotic for F in corollary of Karamata's Theorem}
	\end{equation}
	as $y \to \infty$. Let $\sigma > 0$.
	Then the following statements hold, as $y \to \infty$.
	\begin{enumalph}
	\item If $\sigma > \rho$, then
	\[
	\sum_{n > y} n^{-\sigma} \alpha(n) \sim \frac{\rho y^{-\sigma} F(y)}{\abs{\sigma - \rho}} \sim \frac{\kappa \rho y^{\rho - \sigma} \log^\tau y}{\abs{\sigma - \rho}}.
	\]
	\item If $\rho > \sigma$, then \[
	\sum_{n \leq y} n^{-\sigma} \alpha(n) \sim \frac{\rho y^{-\sigma} F(y)}{\abs{\sigma - \rho}} \sim \frac{\kappa \rho y^{\rho - \sigma} \log^\tau y}{\abs{\sigma - \rho}}.
	\]
	\end{enumalph}
\end{corollary}

\begin{proof}
	Replacing $\alpha$ and $F$ with $-\alpha$ and $-F$ if necessary, we may assume $\kappa > 0$. As a partial sum of an arithmetic function, $F(y)$ is measurable and locally bounded; by \eqref{Equation: asymptotic for F in corollary of Karamata's Theorem}, $F(y)$ is eventually positive. Now for any $\lambda > 0$, we compute
	\[
	\lim_{y \to \infty} \frac{F(\lambda y)}{F(y)} = \lim_{y \to \infty} \frac{\kappa (\lambda y)^\rho \log^\tau (\lambda y)}{\kappa y^\rho \log^\tau y} = \lambda^\rho,
	\]
	so $F$ is regularly varying of index $\rho$.

	Suppose first $\sigma > \rho$. Since
	\[
	y^{-\sigma} F(y) \sim \kappa y^{\rho - \sigma} \log^\tau y \to 0
	\]
	as $y \to \infty$, Abel summation yields
	\[
		\sum_{n > y} n^{-\sigma} \alpha(n) = - y^{-\sigma} F(y) + \sigma \int_y^\infty u^{-\sigma - 1} F(u) \,\mathrm{d}u.
	\]
	Clearly $\sigma + 1 > \rho + 1$, so \Cref{Theorem: Karamata's integral theorem}(a) tells us 
	\[
	\int_y^\infty u^{-\sigma - 1} F(u) \,\mathrm{d}u \sim \frac{y^{-\sigma} F(y)}{\abs{\sigma - \rho}} \sim \frac{\kappa y^{\rho - \sigma} \log^\tau y}{\abs{\sigma - \rho}}
	\]
	and thus
	\[
		\sum_{n > y} n^{-\sigma} \alpha(n) \sim \frac{\rho y^{-\sigma} F(y)}{\abs{\sigma - \rho}}
	\]
	as $y \to \infty$.
	
	The case $\rho > \sigma$ is similar.
\end{proof}

\subsection{Bounding Dirichlet series on vertical lines}

Recall that a complex function $F(s)$ has \defi{finite order} on a domain $D$ if there exists $\xi \in \R_{>0}$ such that 
	\[
	F(s) = O(1 + \abs{t}^\xi)
	\]
	whenever $s = \sigma + i t \in D$. If $F$ is of finite order on a right half-plane, we define
	\[
	\mu_F(\sigma) \colonequals \inf\{\xi \in \bbR_{\geq 0} :F(\sigma + i t) = O(1 + \abs{t}^\xi)\}
	\]
	where the implicit constant depends on $\sigma$ and $\xi$.
	
	Let $L(s)$ be a Dirichlet series with abscissa of absolute convergence $\sigma_a$ and abscissa of convergence $\sigma_c$. 

\begin{theorem}\label{Theorem: absolutely convergent Dirichlet series have mu = 0}
	We have $\mu_L(\sigma)=0$ for all $\sigma > \sigma_a$, and $\mu_L(\sigma)$ is nonincreasing (as a function of $\sigma$) on any region where $L$ has finite order.
\end{theorem}

\begin{proof}
	Tenenbaum \cite[Theorem II.1.21]{Tenenbaum}.
\end{proof}

\begin{theorem}\label{Theorem: Dirichlet series grow slowly on vertical lines}
	Let $\sigma_c < \sigma_0 \leq \sigma_c+1$ and let $\epsilon > 0$. Then uniformly on 
	\[
	\set{s = \sigma + i t \in \bbC : \sigma_0 \leq \sigma \leq \sigma_c + 1, \ \abs{t} \geq 1},
	\]
	we have
	\[
	L(\sigma + i t) = O(t^{1 + \sigma_c - \sigma + \epsilon}).
	\]
\end{theorem}

\begin{proof}
	Tenenbaum \cite[Theorem II.1.19]{Tenenbaum}.
\end{proof}

\begin{corollary}\label{Corollary: bound on mu past abscissa of convergence}
	For all $\sigma > \sigma_c$, we have 
	\[
	\mu_{L}(\sigma) \leq \max(0,1 + \sigma_c - \sigma).
	\]
\end{corollary}

\begin{proof}
	It is well-known that $\sigma_a \leq \sigma_c + 1$, so the claim holds for $\sigma > \sigma_c + 1$ by \Cref{Theorem: absolutely convergent Dirichlet series have mu = 0}. Now for $\sigma_c < \sigma < \sigma_c + 1$, our claim follows by letting $\epsilon \to 0$ in \Cref{Theorem: Dirichlet series grow slowly on vertical lines}.
\end{proof}

\begin{theorem}\label{Theorem: muzeta(sigma)}
	Let $\zeta(s)$ be the Riemann zeta function, and let $\sigma \in \bbR$. We have
	\[
	\mu_\zeta(\sigma) \leq \begin{cases}
	\frac 12 - \sigma, & \text{if} \ \sigma \leq 0; \\
	\frac 12 - \frac{141}{205} \sigma, & \text{if} \ 0 \leq \sigma \leq \frac 12; \\
	\frac{64}{205}(1 - \sigma), & \text{if} \ \frac 12 \leq \sigma \leq 1; \\
	0 & \text{if} \ \sigma \geq 0.
	\end{cases}
	\]
	Moreover, equality holds if $\sigma<0$ or $\sigma>1$.
\end{theorem}

\begin{proof}
	Tenenbaum \cite[page 235]{Tenenbaum} proves the claim when $\sigma<0$ or $\sigma>1$. Now $\mu_\zeta(1/2) \leq 32/205$ by Huxley \cite[Theorem 1]{Huxley}, and our result follows from the subconvexity of $\mu_\zeta$ \cite[Theorem II.1.20]{Tenenbaum}.
\end{proof}

\subsection{A Tauberian theorem}\label{Subsection: Perron's formula}

We now present a Tauberian theorem, due in essence to Landau \cite{Landau1915}. 

\begin{definition}\label{Definition: admissible sequences}
	Let $\parent{\alpha(n)}_{n \geq 1}$ be a sequence with $\alpha(n) \in \R_{\geq 0}$ for all $n$, and let $L_\alpha(s) \colonequals \sum_{n \geq 1} \alpha(n) n^{-s}$.
	We say the sequence $\parent{\alpha(n)}_{n \geq 1}$ is \defi{admissible} with (real) parameters $\parent{\sigma_a, \delta, \xi}$ if the following hypotheses hold:
	\begin{enumroman}
		\item $L_\alpha(s)$ has abscissa of absolute convergence $\sigma_a$.\label{Condition: abscissa of absolute convergence sigma}
		\item The function $L_\alpha(s)/s$ has meromorphic continuation to $\set{s = \sigma + i t \in \C : \sigma > \sigma_a - \delta}$ and only finitely many poles in this region. \label{Condition: meromorphic extension}
		\item For $\sigma > \sigma_a - \delta$, we have $\mu_{L_\alpha}(\sigma) \leq \xi$.
		\label{Condition: mu is bounded}
	\end{enumroman}
\end{definition}

	If $\parent{\alpha(n)}_n$ is admissible, let $s_1, \dots, s_r$ denote the poles of $L_\alpha(s)/s$ with real part greater than $\sigma_a - \delta/(\xi + 2)$.

The following theorem is essentially an application of Perron's formula, which is itself an inverse Mellin transform.

\begin{theorem}[Landau's Tauberian Theorem]\label{Theorem: Landau's Tauberian theorem}
	Let $\parent{\alpha(n)}_{n \geq 1}$ be an admissible sequence (\textup{\Cref{Definition: admissible sequences}}), and write $N_\alpha(X) \colonequals \sum_{n \leq X} \alpha(n)$. Then for all $\epsilon>0$,
		\[
		N_\alpha(X) = \sum_{j = 1}^r \res_{s=s_j}\parent{\frac{L_\alpha(s) X^{s}}{s}} + O\!\parent{X^{\sigma_a - \frac{\delta}{\floor{\xi} + 2} + \epsilon}},
		\]
		where the main term is a sum of residues and the implicit constant depends on $\epsilon$.
\end{theorem}

\begin{proof}
	See Roux \cite[Theorem 13.3, Remark 13.4]{Roux}.
\end{proof}

\begin{remark}
	Landau's original theorem \cite{Landau1915} was fitted to a more general context, and allowed sums of the form
	\[
	\sum_{n \geq 1} \alpha(n) \ell(n)^{-s}
	\]
	as long as $\parent{\ell(n)}_{n \geq 1}$ was increasing and tended to $\infty$. Landau also gave an explicit expansion of
	\[
	\res_{s=s_j}\parent{\frac{L_\alpha(s) X^{s}}{s}}
	\]
	in terms of the Laurent series expansion for $L_\alpha(s)$ around $s = s_j$. However, Landau also required that $L_\alpha(s)$ has a meromorphic continuation to all of $\bbC$, and Roux \cite[Theorem 13.3, Remark 13.4]{Roux} relaxes this assumption.
\end{remark}

Let $d(n)$ denote the number of divisors of $n$, and let $\omega(n)$ denote the number of distinct prime divisors of $n$.
Theorem \ref{Theorem: Landau's Tauberian theorem} has the following easy corollary.

\begin{corollary}\label{Corollary: sum of 2^omega(n) and d(n)^2}
	We have 
	\[
	\sum_{n \leq y} 2^{\omega(n)} = \frac{y \log y}{\zeta(2)} + O(y) \quad \text{and} \quad \sum_{n \leq y} d(n)^2 = \frac{y \log^3 y}{6 \zeta(2)} + O(y \log^2 y).
	\]
	as $y \to \infty$.
\end{corollary}

\begin{proof}
	 Recall that 
\[
\frac{\zeta(s)^2}{\zeta(2s)} = \sum_{n \geq 1} \frac{2^{\omega(n)}}{n^s} \ \text{and} \ \frac{\zeta(s)^4}{\zeta(2s)} = \sum_{n \geq 1} \frac{d(n)^2}{n^s}.
\]
	It is straightforward to verify that $\parent{2^{\omega(n)}}_{n \geq 1}$ and $\parent{d(n)^2}_{n \geq 1}$ are both admissible with parameters $(1, 1/2, 1/3)$. We apply Theorem \ref{Theorem: Landau's Tauberian theorem} and discard lower-order terms to obtain the result. 
\end{proof}

\begin{remark}
	\Cref{Theorem: Landau's Tauberian theorem} furnishes lower order terms for the sums $\sum_{n \leq y} 2^{\omega(n)}$ and $\sum_{n \leq y} d(n)^2$, and even better estimates are known (e.g.\ Tenenbaum \cite[Exercise I.3.54]{Tenenbaum} and Zhai \cite[Corollary 4]{Zhai}), but \Cref{Corollary: sum of 2^omega(n) and d(n)^2} suffices for our purposes and illustrates the use of \Cref{Theorem: Landau's Tauberian theorem}.
\end{remark}

\section{Estimates for twist classes} \label{Section: Estimating twN(X)}

In this section, we decompose $\twistN(X)$, counting the number of twist minimal elliptic curves over $\Q$ admitting a $7$-isogeny \eqref{eqn:NQxwist} in terms of progressively simpler functions. We then estimate those simple functions, and piece these estimates together until we arrive at an estimate for $\twistN(X)$; the main result is \Cref{Theorem: asymptotic for twN(X)}, which proves \Cref{Intro Theorem: asymptotic for twN(X)}.

\subsection{Decomposition and outline} \label{sec:decomp}

We establish some notation for brevity and ease of exposition. Suppose $\parent{\alpha(X; n)}_{n \geq 1}$ is a sequence of real-valued functions, and $\phi : \bbR_{>0} \to \bbR_{>0}$. We write
\[
\sum_{n \geq 1} \alpha(X; n) = \sum_{n \ll \phi(X)} \alpha(X; n)
\]
if there is a positive constant $\kappa$ such that for all $X \in \bbR_{> 0}$ and all $n > \kappa \phi(X)$, we have $\alpha(X; n) = 0$.

The function $\twistN(X)$ is difficult to understand chiefly because of the twist minimality defect. Fortunately, the twist minimality defect cannot get too large relative to $X$ (see \Cref{Corollary: bound twist defect in terms of twist height}). So we partition our sum based on the value of $\twistdefect(A(a,b), B(a,b))$ in terms of the parametrization provided in \cref{sec:7isog}.

	For $e \geq 1$, let $\twistN(X; e)$ denote the number of pairs $(a, b) \in \bbZ^2$ with
	\begin{itemize}
		\item $(a, b)$ groomed,
		\item $\twistheight(A(a, b), B(a, b)) \leq X$, and
		\item $\twistdefect(A(a, b), B(a, b)) = e$.
	\end{itemize}

By \eqref{eqn:twistheightcalc} and \Cref{Corollary: bound twist defect in terms of twist height}, we have
\begin{equation}
\twistN(X) = \sum_{e \ll X^{1/12}} \twistN(X; e);\label{Equation: twN(X) in terms of twN(X; e)}
\end{equation}
more precisely, we can restrict our sum to 
\[
e \leq \frac{3^{5/4} \cdot 7^{9/2}}{2^{1/6}} \cdot X^{1/12}.
\]

Determining when an integer $e$ \emph{divides} $\twistdefect(A, B)$ is easier than determining when $e$ \emph{equals} $\twistdefect(A, B)$, so we also let $\cM(X; e)$ denote the number of pairs $(a, b) \in \bbZ^2$ with
\begin{itemize}
	\item $(a, b)$ groomed,
	\item $\rawheight(A(a,b),B(a,b)) \leq X$; 
	\item $e \mid \twistdefect(A(a, b), B(a, b))$;
\end{itemize}
	Note that the points counted by $\twistN(X; e)$ have \emph{twist} height bounded by $X$, but the points counted by $\cM(X; e)$ have only the function $H$ bounded by $X$.

	\Cref{Theorem: Controlling size of twist minimality defect} and the M\"{o}bius sieve yield
	\begin{equation}
	\twistN(X; e) = \sum_{f \ll \frac{X^{1/18}}{e^{2/3}}} \mu(f) \cM(e^6 X; ef);\label{Equation: twistN(X; e) in terms of M(X; e)}
	\end{equation}
	more precisely, we can restrict our sum to
	\[
	f \leq \frac{3^{1/2} 7^2}{2^{1/9}} \cdot \frac{X^{1/18}}{e^{2/3}}.
	\] 

	In order to estimate $\cM(X; e)$, we further unpack the groomed condition on pairs $(a, b)$. We therefore let $\cM(X; d, e)$ denote the number of pairs $(a, b) \in \bbZ^2$ with
	\begin{itemize}
		\item $\gcd(da, db, e) = 1$ and $b > 0$;
		\item $\rawheight(A(d a, d b), B(d a, db)) \leq X$;
		\item $e \mid \twistdefect(A(d a, d b), B(da, db))$;
		\item $(a, b) \neq (-7, 1)$.
	\end{itemize}
By \Cref{Theorem: Controlling size of twist minimality defect}, and because $\rawheight(A(a, b), B(a, b))$ is homogeneous of degree 12, another M\"{o}bius sieve yields
\begin{equation}
	\cM(X; e) = \sum_{\substack{d \ll X^{1/12} \\ \gcd(d, e) = 1}} \mu(d) \cM(X; d, e);\label{Equation: cM(X;e) in terms of cM(X; d, e)}
\end{equation}	
more precisely, we can restrict our sum to
	\[
	d \leq \frac{1}{2^{1/6} \cdot 3^{1/4}} \cdot X^{1/12}.
	\] 

Before proceeding, we now give an outline of the argument used in this section. In \Cref{Lemma: asymptotic for M(X; e)}, we use the Principle of Lipschitz to estimate $\cM(X; d, e)$, then piece these estimates together using \eqref{Equation: cM(X;e) in terms of cM(X; d, e)} to estimate $\cM(X; e)$. Heuristically,
\begin{equation}
\cM(X; d, e) \sim \frac{R T(e) X^{1/6}}{d^2 e^3} \prod_{\ell \mid e} \parent{1 - \frac{1}{\ell}}
\end{equation}
(where $R$ is the area of \eqref{eqn: R(X)} and $T$ is the arithmetic function investigated in \Cref{Lemma: bound on T(e)})
by summing over the congruence classes modulo $e^3$ that satisfy $e \mid \twistdefect(A(d a, d b), B(da, db))$. Then \eqref{Equation: cM(X;e) in terms of cM(X; d, e)} suggests
\begin{equation}
\cM(X; e) \sim \frac{R T(e) X^{1/6}}{\zeta(2) e^3 \prod_{\ell \mid e} \parent{1 + \frac{1}{\ell}}}.\label{Equation: Main term for cM(X; e)}
\end{equation}

To go further, we substitute \eqref{Equation: twistN(X; e) in terms of M(X; e)} into \eqref{Equation: twN(X) in terms of twN(X; e)}, and let $n = e f$ to obtain
\begin{equation}
\twistN(X) = \sum_{n \ll X^{1/12}} \sum_{e \mid n} \mu\parent{n/e} \cM(e^6 X; n).\label{Equation: twN(X) in terms of M(X; e), reorganized}
\end{equation}
This is the core identity that, in concert with the Principle of Lipschitz, enables us to estimate $\twistN(X)$.

Substituting \eqref{Equation: Main term for cM(X; e)} into \eqref{Equation: twN(X) in terms of M(X; e), reorganized}, and recalling $\varphi(n) = \sum_{e \mid n} \mu(n/e) e$, we obtain the heuristic estimate
\begin{equation}
\twistN(X) \sim \frac{Q R X^{1/6}}{\zeta(2)},
\end{equation}
where
\begin{equation}
Q \colonequals \sum_{n \geq 1} \frac{T(n) \varphi(n) }{n^3 \prod_{\ell \mid n} \parent{1 + \frac{1}{\ell}}}.
\end{equation}

	To make this estimate for $\twistN(X)$ rigorous, and to get a better handle on the size of order of growth for its error term, we now decompose \eqref{Equation: twN(X) in terms of M(X; e), reorganized} based on the size of $n$ into two pieces:
	\begin{equation}
\begin{aligned}
	\twistNly(X) &\colonequals \sum_{n \leq y} \sum_{e \mid n} \mu\parent{\frac ne} \cM(e^6 X; n), \\
	\twistNgy(X) &\colonequals \sum_{n > y} \sum_{e \mid n} \mu\parent{n/e} \cM(e^6 X; n).
\end{aligned}
\end{equation}
	By definition, we have
	\[ \twistN(X) = \twistNly(X) + \twistNgy(X). \] 
	We then estimate $\twistNly(X)$ in \Cref{Lemma: asymptotic for twN<=y(X)}, and treat $\twistNgy(X)$ as an error term which we bound in \Cref{Lemma: bound on twN>y(X)}. Setting the error from our estimate equal to the error arising from $\twistNgy(X)$, we obtain Theorem \ref{Theorem: asymptotic for twN(X)}.

In the remainder of this section, we follow the outline suggested here by successively estimating $\cM(X; d, e)$, $\cM(X; e)$, $\twistNly(X)$, $\twistNgy(X)$, and finally $\twistN(X)$.

\subsection{Asymptotic estimates}

We first estimate $\cM(X; d, e)$ and $\cM(X; e)$.

\begin{lemma}\label{Lemma: asymptotic for M(X; e)}
The following statements hold.
\begin{enumalph}
\item	If $\gcd(d, e) > 1$, then $\cM(X; d, e) = 0$. Otherwise, we have 
	\[
	\cM(X; d, e) = \frac{R T(e) X^{1/6}}{d^2 e^3} \prod_{\ell \mid e} \parent{1 - \frac{1}{\ell}} + O\parent{\frac{T(e) X^{1/12}}{d}}.
	\]
	where $R$ is the area of \eqref{eqn: R(X)}.
\item We have
	\[
	\cM(X; e) = \frac{R T(e) X^{1/6}}{\zeta(2) e^3 \prod_{\ell \mid e} \parent{1 + \frac{1}{\ell}}} + O(T(e) X^{1/12} \log X).
	\]
\end{enumalph}
In both cases, the implied constants are independent of $d$, $e$, and $X$.
\end{lemma}

\begin{proof}
	We begin with (a) and examine the summands $\cM(X; d, e)$. If $d$ and $e$ are not coprime, then $\cM(X; d, e) = 0$ because $\gcd(da, db, e) \geq \gcd(d, e) > 1$. On the other hand, if $\gcd(d, e) = 1$, we have a bijection from the pairs counted by $\cM(X; 1, e)$ to the pairs counted by $\cM(d^{12} X; d, e)$ given by $(a, b) \mapsto (d a, d b)$.

Combining \Cref{Lemma: bound on T(e)}(a) and \Cref{Corollary: Estimates for L(X)}, we have
\begin{equation}
\begin{aligned}
	\cM(X; 1, e) &= \sum_{(a_0, b_0) \in \widetilde{\calT}(e)} \#\{(a, b) \in \calR(X) \cap \bbZ^2 : (a, b) \equiv (a_0, b_0) \psmod {e^3}, (a, b) \neq (-7, 1) \}\\
	&= \varphi(e^3) T(e) \parent{\frac{R X^{1/6}}{e^6} + O\parent{\frac{X^{1/12}}{e^3}}} \\
	&= \frac{R T(e) X^{1/6}}{e^3} \prod_{\ell \mid e} \parent{1 - \frac{1}{\ell}} + O(T(e) X^{1/12}),
\end{aligned}
\end{equation}
and thus
\[
	\cM(X; d, e) = \frac{R T(e) X^{1/6}}{d^2 e^3} \prod_{\ell \mid e} \parent{1 - \frac{1}{\ell}} + O\parent{\frac{T(e) X^{1/12}}{d}}.
\]

For part (b), we compute
\begin{equation} \label{eqn:cmXe}
\begin{aligned}
	\cM(X; e) &= \sum_{\substack{d \ll X^{1/12} \\ \gcd(d, e) = 1}} \mu(d) \cM(X; d, e) \\
	&= \sum_{\substack{d \ll X^{1/12} \\ \gcd(d, e) = 1}} \mu(d) \parent{\frac{T(e) R X^{1/6}}{d^2 e^3} \prod_{\ell \mid e} \parent{1 - \frac{1}{\ell}} + O\parent{T(e) \frac{X^{1/12}}{d}}} \\
	&= \frac{R T(e) X^{1/6}}{e^3} \prod_{\ell \mid e} \parent{1 - \frac{1}{\ell}} \sum_{\substack{d \ll X^{1/12} \\ \gcd(d, e) = 1}} \frac{ \mu(d)}{d^2} + O\parent{T(e) X^{1/12} \sum_{\substack{d \ll X^{1/12} \\ \gcd(d, e) = 1}} \frac 1d}.
	\end{aligned}
	\end{equation}
Plugging the straightforward estimates
\begin{equation}
\sum_{\substack{d \ll X^{1/12} \\ \gcd(d, e) = 1}} \frac{ \mu(d)}{d^2} = \frac{1}{\zeta(2)} \prod_{\ell \mid e} \parent{1 - \frac{1}{\ell^2}}\inv + O(X^{-1/12})
\end{equation}
and
\[ 
\sum_{\substack{d \leq X^{1/12}}} \frac 1d = \frac{1}{12}\log X+ O(1) \]
into \eqref{eqn:cmXe} then simplifies to give
\begin{equation}
\begin{aligned}
\cM(X;e)	
	&= \frac{R T(e) X^{1/6}}{\zeta(2) e^3 \prod_{\ell \mid e} \parent{1 + \frac{1}{\ell}}} + O(T(e) X^{1/12} \log X)
\end{aligned}
\end{equation}
proving (b).
\end{proof}

We are now in a position to estimate $\twistNly(X)$.

\begin{prop}\label{Lemma: asymptotic for twN<=y(X)}
	Suppose $y \ll X^{\frac{1}{12}}$. Then
	\[
	\twistNly(X) = \frac{Q R X^{1/6}}{\zeta(2)} + O\parent{\max\parent{\frac{X^{1/6} \log y}{y}, X^{1/12} y^{3/2} \log X \log^3 y }}
	\]
	where
	\[
	Q \colonequals \sum_{n \geq 1} \frac{\varphi(n) T(n)}{n^3 \prod_{\ell \mid n} \parent{1 + \frac{1}{\ell}}} = Q_3 Q_7 \prod_{\substack{p \neq 7 \ \textup{prime} \\ p \equiv 1 \psmod {3}}} \parent{1 + \frac{2}{(p+1)^2}}, 
	\]
	and $Q_3 = 13/6$, $Q_7=63/8$.
\end{prop}

\begin{proof}
	Substituting the asymptotic for $\cM(X; e)$ from \Cref{Lemma: asymptotic for M(X; e)} into the defining series for $\twistNly(X)$, we have
	\[
		\twistNly(X) = \sum_{n \leq y} \sum_{e \mid n} \mu\parent{n/e} \parent{\frac{R T(n) e X^{1/6}}{\zeta(2) n^3 \prod_{\ell \mid n} \parent{1 + \frac{1}{\ell}}} + O\parent{T(n) e^{1/2} X^{1/12} \log (e^6 X)}}.
	\]
	
	We handle the main term and the error of this expression separately. For the main term, we have
	\begin{equation}
	\begin{aligned}
	\sum_{n \leq y} \sum_{e \mid n} \mu\parent{n/e} \parent{\frac{R T(n) e X^{1/6}}{\zeta(2) n^3 \prod_{\ell \mid n} \parent{1 + \frac{1}{\ell}}}} &= \frac{R X^{1/6}}{\zeta(2)} \sum_{n \leq y} \frac{T(n)}{n^3 \prod_{\ell \mid n} \parent{1 + \frac{1}{\ell}}} \sum_{e \mid n} \mu\parent{n/e} e \\
	&= \frac{R X^{1/6}}{\zeta(2)} \sum_{n \leq y} \frac{\varphi(n) T(n)}{n^3 \prod_{\ell \mid n} \parent{1 + \frac{1}{\ell}}}.
	\end{aligned}
	\end{equation}
	By \Cref{Lemma: bound on T(e)}(d), we see
	\[
	\frac{\varphi(n) T(n)}{n^3 \prod_{\ell \mid n} \parent{1 + \frac{1}{\ell}}} = O\parent{\frac{2^{\omega(n)}}{n^2}}.
	\]
	
	By \Cref{Corollary: tail of sum of f(n)/n^sigma} and \Cref{Corollary: sum of 2^omega(n) and d(n)^2}, we have
	\[
	\sum_{n > y} \frac{2^{\omega(n)}}{n^2} \sim \frac{\log y}{\zeta(2) y}
	\]
	as $y \to \infty$. \textit{A fortiori,}
	\[
	\sum_{n > y} \frac{\varphi(n) T(n)}{n^3 \prod_{\ell \mid n} \parent{1 + \frac{1}{\ell}}} = O\parent{\sum_{n > y} \frac{2^{\omega(n)}}{n^2}} = O\parent{\frac{\log y}{y}},
	\]
	so the series
	\begin{equation}
	\sum_{n \geq 1} \frac{\varphi(n) T(n)}{n^3 \prod_{\ell \mid n} \parent{1 + \frac{1}{\ell}}} = Q \label{Equation: sum for Q}
	\end{equation}
	is absolutely convergent, and 
	\begin{equation}
	\begin{aligned}
	\sum_{n \leq y} \sum_{e \mid n} \mu\parent{n/e} \parent{\frac{R T(n) e X^{1/6}}{\zeta(2) n^3 \prod_{\ell \mid n} \parent{1 + \frac{1}{\ell}}}} &= \frac{R X^{1/6}}{\zeta(2)} \parent{Q - O\parent{\frac{\log y}{y}}} \\
	&= \frac{Q R X^{1/6}}{\zeta(2)} + O\parent{\frac{X^{1/6} \log y}{y}}.
	\end{aligned}
	\end{equation}
	
	As the summands of \eqref{Equation: sum for Q} constitute a nonnegative multiplicative arithmetic function, we can factor $Q$ as an Euler product. For $p$ prime, \Cref{Lemma: bound on T(e)} yields
	\begin{equation}
	Q_p \colonequals \sum_{a \geq 0} \frac{\varphi(p^a) T(p^a)}{p^{3a} \prod_{\ell \mid p} \parent{1 + \frac{1}{\ell}}} = \begin{cases}
	1 + \displaystyle{\frac{2}{p^2 + 1}}, & \textup{if $p \equiv 1 \psmod{3}$ and $p \neq 7$;} \\
	13/6, & \text{if $p=3$;} \\
	63/8, & \text{if $p=7$;} \\ 
	1 & \text{else}.
	\end{cases}
	\end{equation}
	Thus
	\begin{equation}
	Q = \prod_{\textup{$p$ prime}} Q_p = Q_3 Q_7 \prod_{\substack{p \neq 7 \ \textup{prime} \\ p \equiv 1 \psmod {3}}} \parent{1 + \frac{2}{p^2 + 1}}. \label{Equation: product for Q}
	\end{equation}
	
	We now turn to the error term. Since $y \ll X^{1/12}$, for $e \leq y$ we have $\log (e^6 X) \ll \log X$. Applying \Cref{Lemma: bound on T(e)}(d), we obtain
	\begin{align}
		\sum_{n \leq y} \sum_{e \mid n} \mu\parent{n/e} O\parent{T(n) e^{1/2} X^{1/12} \log \parent{e^6 X}} &= O\parent{X^{1/12} \log X \sum_{n \leq y} T(n) \sum_{e \mid n} \abs{\mu\parent{\frac{n}{e}}} e^{1/2} } \nonumber \\
		&= O\parent{X^{1/12} \log X \sum_{n \leq y} 2^{2 \omega(n)} n^{1/2} }. \label{Equation: partial simplification twN<=y(X)}
	\end{align}
	\Cref{Corollary: tail of sum of f(n)/n^sigma} and \Cref{Corollary: sum of 2^omega(n) and d(n)^2}, together with the trivial inequality $2^{2 \omega(n)} \leq d(n)^2$, yield
	\begin{equation}
	\sum_{n \leq y} 2^{2 \omega(n)} n^{1/2} = O(y^{3/2} \log^3 y). \label{Equation: approximation 2^2 omega(n) n^1/2}
	\end{equation}
	Substituting \eqref{Equation: approximation 2^2 omega(n) n^1/2} into \eqref{Equation: partial simplification twN<=y(X)} gives our desired result.
\end{proof}

We now bound $\twistNgy(X)$.

\begin{lemma}\label{Lemma: bound on twN>y(X)}
	We have
	\[
	\twistNgy(X) = O\parent{\frac{X^{1/6} \log^3 y}{y}}.
	\]
\end{lemma}

\begin{proof}
	We have
	\begin{align}
	\twistNgy(X) &= \sum_{n > y} \sum_{e \mid n} \mu\parent{n/e} \cM(e^6 X; n)
	\leq \sum_{n > y} 2^{\omega(n)} \cM(n^6 X; n).\label{Equation: 1st partial simplification of twNgy(X)}
	\end{align}
	Write $n = 3^v 7^w n\prm$ where $\gcd(n\prm, 3) = \gcd(n\prm, 7) = 1$. We define 
	\[
	n_0 \coloneqq 3^{\max(v - 1, 0)} 7^{\max(w - 3, 0)} n\prm,
	\]
	so
	\[
	\frac{n}{3 \cdot 7^3} \leq n_0 \leq n.
	\]
	Let $(a, b) \in \bbZ^2$ be a groomed pair. By \Cref{Theorem: Controlling size of twist minimality defect}(a), $H(A(a, b), B(a, b)) \leq n^6 X$ implies $108 C(a, b)^6 \leq n^6 X$, and by \Cref{Theorem: Controlling size of twist minimality defect}(b), $n \mid \twistdefect(A(a, b), B(a, b))$ implies $n_0^3 \mid C(a, b)^3$. Thus
	\begin{equation}
	\cM(n^6 X; n) \leq \#\set{(a, b) \in \bbZ^2 \ \text{groomed} : 108 C(a, b)^6 \leq n^6 X, \ n_0^3 \mid C(a, b)}.
	\end{equation}
	Recalling \eqref{Equation: c(m)} and \Cref{Lemma: c(m) is multiplicative etc}(c), we deduce
	\[
	\cM(n^6 X; n) \leq \sum_{m \ll X^{1/6}/n^2} c(n_0^3 m) \leq 3 \cdot 2^{\omega(n_0) - 1} \sum_{m \ll X^{1/6}/n^2} c(m).
	\]
	But $2^{\omega(n)} \leq 4 \cdot 2^{\omega(n_0)}$, so by \Cref{Corollary: Bound on partial sums of c(m)}, we have
	\[
	\cM(n^6 X; n) = O\parent{\frac{2^{\omega(n)} X^{1/6}}{n^2}},
	\]
	and substituting this expression into \eqref{Equation: 1st partial simplification of twNgy(X)} yields
	\begin{equation}
		\twistNgy(X) = O\parent{\sum_{n > y} \frac{\parent{2^{\omega(n)}}^2 X^{1/6}}{n^2}} = O\parent{X^{1/6} \sum_{n > y} \frac{2^{2 \omega(n)}}{n^2}}.\label{Equation: 2nd partial simplification of twNgy(X)}
	\end{equation}
	As in the proof of \Cref{Lemma: asymptotic for twN<=y(X)}, combining \Cref{Corollary: tail of sum of f(n)/n^sigma} and \Cref{Corollary: sum of 2^omega(n) and d(n)^2} together with the trivial inequality $2^{2 \omega(n)} \leq d(n)^2$ yields 
	\begin{equation}
	\sum_{n > y} \frac{2^{2 \omega(n)}}{n^2} = O\parent{\frac{\log^3 y}{y}}. \label{Equation: approximation 2^2 omega(n) / n^2}
	\end{equation}
	Substituting \eqref{Equation: approximation 2^2 omega(n) / n^2} into \eqref{Equation: 2nd partial simplification of twNgy(X)} gives our desired result.	
\end{proof}

We are now in a position to prove \Cref{Intro Theorem: asymptotic for twN(X)}, which we restate here with the notations we have established.

\begin{theorem}\label{Theorem: asymptotic for twN(X)}
	We have 
	\[
	\twistN(X) = \frac{Q R X^{1/6}}{\zeta(2)} + O(X^{2/15} \log^{17/5} X),
	\]
	where
	\[
	Q = \sum_{n \geq 1} \frac{\varphi(n) \cT (n)}{n^3 \prod_{\ell \mid n} \parent{1 + 1/\ell}},
	\]
	and $R$ is the area of the region
	\[
	\calR(1) = \set{(a, b) \in \bbR^2 : \rawheight(A (a, b), B (a, b)) \leq 1, b \geq 0}.
	\]
\end{theorem}

\begin{proof}
	Let $y$ be a positive quantity with $y \ll X^{1/12}$; in particular, $\log y \ll \log X$. \Cref{Lemma: asymptotic for twN<=y(X)} and \Cref{Lemma: bound on twN>y(X)} together tell us
	\begin{equation}
	\twistN(X) = \frac{Q R X^{1/6}}{\zeta(2)} + O\parent{\max\parent{\frac{X^{1/6} \log^3 y}{y}, X^{1/12} y^{3/2} \log X \log^3 y}}.
	\end{equation}

	We let $y = X^{1/30}/\log^{2/5} X$, so
	\begin{equation}
	\frac{X^{1/6} \log^3 y}{y} \asymp X^{1/12} y^{3/2} \log X \log^3 y \asymp X^{2/15} \log^{17/5} X,
	\end{equation}
	and we conclude
	\[
	\twistN(X) = \frac{Q R X^{1/6}}{\zeta(2)} + O(X^{2/15} \log^{17/5} X)
	\]
	as desired.
\end{proof}

\subsection{$L$-series}

To conclude, we set up the next section by interpreting
\Cref{Theorem: asymptotic for twN(X)} in terms of Dirichlet series. Let
	\begin{equation}\label{Equation: twisth(n)}
	\twisth(n) \colonequals \# \set{(a, b) \in \bbZ^2 \textup{ groomed : $\twht(A(a,b),B(a,b)) = n$}}
	\end{equation}
	and define 
	\begin{equation}
	\twistL(s) \colonequals \sum_{n \geq 1} \frac{\twisth(n)}{n^s} 
	\end{equation}
	wherever this series converges. Then $\twistN(X) = \sum_{n \leq X} \twisth(n)$, and conversely we have 
$\twistL(s) = \int_0^\infty u^{-s} \,\mathrm{d}\twistN(u) $.

\begin{cor}\label{Corollary: twL(s) has a meromorphic continuation}
	The Dirichlet series $\twistL(s)$ has abscissa of (absolute) convergence $\sigma_a=\sigma_c = 1/6$ and has a meromorphic continuation to the region
	\begin{equation}
	\set{s = \sigma + i t \in \bbC : \sigma > 2/15}. \label{Equation: domain of twistL(s)}
	\end{equation}
	Moreover, $\twistL(s)$ has a simple pole at $s = 1/6$ with residue 
	\[ \res_{s=\frac{1}{6}} \twistL(s) = \frac{QR}{6\zeta(2)} \]
	and is holomorphic elsewhere on the region \eqref{Equation: domain of twistL(s)}.
\end{cor}

\begin{proof}
	Let $s = \sigma + i t \in \bbC$ be given with $\sigma > 1/6$. Abel summation yields
	\begin{equation}
	\begin{aligned}
		\sum_{n \leq X} \twisth(n) n^{-s} &= \twistN(X) X^{-s} + s \int_1^X \twistN(u) u^{-s-1} \,\mathrm{d}u \\
		&= O\parent{X^{1/6 - \sigma} + s \int_1^X u^{- 5/6 - \sigma} \,\mathrm{d}u};
	\end{aligned}
	\end{equation}
	as $X \to \infty$ the first term vanishes and the integral converges. Thus, when $\sigma > 1/6$,
	\[
	\sum_{n \geq 1} \twisth(n) n^{-s} = s \int_1^\infty \twistN(u) u^{-1-s}\,\mathrm{d}u
	\]
	and this integral converges. A similar argument shows that the sum defining $\twistL(s)$ diverges when $\sigma < 1/6$. We have shown $\sigma_c = 1/6$ is the abscissa of convergence for $\twistL(s)$, but as $\twisth(n) \geq 0$ for all $n$, it is also the abscissa of \emph{absolute} convergence $\sigma_a=\sigma_c$.
	
	Now define $\twistLR(s)$ so that
	\begin{equation}
		\twistL(s) = \frac{QR}{\zeta(2)} \zeta(6s) + \twistLR(s).\label{Equation: Defining twLR}
	\end{equation}		
	Abel summation and the substitution $u \mapsto u^{1/6}$ yields for $\sigma>1$
	\[
	\zeta(6s) = s \int_1^\infty \floor{u^{1/6}} u^{- 1 - s}\,\textrm{d}u = s \int_1^\infty \parent{u^{1/6} + O(1)} u^{- 1 - s} \,\textrm{d}u.
	\]
	Let 
	\[
	\delta(n) \colonequals \begin{cases} 1, & \text{if} \ n = k^6 \ \text{for some} \ k \in \bbZ; \\
	0, & \text{else.}
	\end{cases}
	\]
	Then
	\begin{equation}
	\begin{aligned}
	\twistLR(s) &= \sum_{n \geq 1} \parent{\twisth(n) - \frac{QR}{\zeta(2)}\delta(n)} n^{-s}\\
	&= s \int_1^\infty \parent{\twistN(u) - \frac{QR}{\zeta(2)} \floor{u^{1/6}}} u^{-1-s} \,\textrm{d}u \label{Equation: twistLR(s) as an integral}
	\end{aligned}
	\end{equation}
	when $\sigma > 1/6$. But then for any $\epsilon > 0$,
	\begin{equation}
	\twistN(u) - \frac{QR}{\zeta(2)} \floor{u^{1/6}} = O(u^{2/15 + \epsilon}) \label{Equation: twistN(t) - C floor(t^1/6)}
	\end{equation}
	by \Cref{Theorem: asymptotic for twN(X)}. Substituting \eqref{Equation: twistN(t) - C floor(t^1/6)} into \eqref{Equation: twistLR(s) as an integral}, we obtain
	\begin{align}
	\twistLR(s) = s \int_1^\infty \parent{\twistN(u) - \frac{QR}{\zeta(2)} \floor{u^{1/6}}} u^{-1-s}\,\textrm{d}u &= O\left(s \int_1^\infty u^{-13/15 - \sigma + \epsilon} \,\textrm{d}u\right) \label{Equation: bound on twLR}
	\end{align}
	where the integral converges whenever $\sigma > 2/15 + \epsilon$. Letting $\epsilon \to 0$, we obtain an analytic continuation of $\twistLR(s)$ to the region \eqref{Equation: domain of twistL(s)}. 
	
	At the same time, $\zeta(6s)$ has meromorphic continuation to $\C$ with a simple pole at $s=1/6$ with residue $1/6$. Thus looking back at \eqref{Equation: Defining twLR}, we find that 
	\[
	\twistL(s) = \frac{QR}{\zeta(2)} \zeta(6s) + s \int_1^\infty \parent{\twistN(u) - \frac{QR}{\zeta(2)} \floor{u^{1/6}}} u^{-1-s} \,\textrm{d}u
	\]
	when $\sigma > 1/6$, but in fact the right-hand side of this equality defines a meromorphic function on the region \eqref{Equation: domain of twistL(s)} with a simple pole at $s = 1/6$ and no other poles. Our claim follows. 
\end{proof}

\section{Estimates for rational isomorphism classes}\label{Section: Working over the rationals}

In \cref{Section: Estimating twN(X)}, we counted the number of elliptic curves over $\Q$ with a $7$-isogeny up to isomorphism over $\Qalg$ (\Cref{Theorem: asymptotic for twN(X)}). In this section, we count all isomorphism classes over $\Q$ by enumerating over twists using a Tauberian theorem (\Cref{Theorem: Landau's Tauberian theorem}).

\subsection{Setup}

	Breaking up the sum \eqref{eqn:NQX}, let
	\begin{equation}
	\hQ(n) \colonequals \#\{(a, b, c) \in \bbZ^3 : \textup{$(a, b)$ groomed, $c$ squarefree, $\hht(c^2 A(a,b),c^3 B(a,b)) = n$}\} \label{Definition: h(n), N(n), L(n)}.
	\end{equation}
	Then $\hQ(n)$ counts the number of elliptic curves $E \in \scrE$ of height $n$ that admit a $7$-isogeny \eqref{eqn:NQx} 
	and 
	\begin{equation}
	\NQ(X) = \sum_{n \leq X} \hQ(n).\label{Equation: NQ(X) in terms of h7}
	\end{equation}
	We also let 
	\begin{equation}
	\LQ(s) \colonequals \sum_{n \geq 1} \frac{\hQ(n)}{n^s}
	\end{equation} 
	wherever this sum converges.

\begin{theorem}\label{Theorem: relationship between twistL(s) and L(s)}
The following statements hold.
\begin{enumalph}
\item	We have
	\[
	\hQ(n) = 2 \sum_{c^6 \mid n} \abs{\mu(c)} \twisth(n/c^6)
	\]
	\item For $s = \sigma + i t \in \bbC$ with $\sigma > 1/6$ we have
	\begin{equation}
	\LQ(s) = \frac{2 \zeta(6s) \twistL(s)}{\zeta(12s)}
	\end{equation}
	with absolute convergence on this region. 
	\item The Dirichlet series $\LQ(s)$ has a meromorphic continuation to the region \eqref{Equation: domain of twistL(s)} with a double pole at $s = 1/6$ and no other singularities on this region. 
	\item The Laurent expansion for $\LQ(s)$ at $s = 1/6$ begins
	\begin{equation}
	\LQ(s) = \frac{1}{3 \zeta(2)^2} \parent{\frac{QR}{6} \parent{s - \frac 16}^{-2} + \parent{\zeta(2) \ell_0 + Q R \parent{\gamma - \frac{2 \zeta\prm(2)}{\zeta(2)}}} \parent{s - \frac{1}{6}}\inv + O(1)},
	\end{equation}
	where
	\begin{equation}
	\ell_0 \colonequals \frac{Q R \gamma}{\zeta(2)} + \frac {1}6 \int_1^\infty \parent{\twistN(u) - \frac{QR}{\zeta(2)} \floor{u^{1/6}}} u^{-7/6} \,\mathrm{d}u\label{Equation: Definition of ell0}
	\end{equation}
	is the constant term of the Laurent expansion for $\twistL(s)$ around $s = 1/6$.
	\end{enumalph}
\end{theorem}

\begin{proof}
Part (a) follows directly from \Cref{lem:7isog-param} and \eqref{Equation: quadratic twists multiply height by c^6} and is something true independent of the parametrization: to count all elliptic curves up to isomorphism by height, it suffices to count them as twists of only the twist minimal curves.  More precisely, from \eqref{Equation: quadratic twists multiply height by c^6} we have $\hht(c^2 A(a,b), c^3B(a,b))=n$ with $c$ squarefree if and only if $\twistheight(A(a,b),B(a,b))=n/(c')^6$ where $c' \colonequals ce/\gcd(c,e)^2$ and $e \colonequals \tmd(A(a,b),B(a,b))$ and $c'$ squarefree.  Thus
\[ \hQ(n) = \sum_{\substack{\textup{$c'$ squarefree} \\ (c')^6 \mid n}} \twisth(n/(c')^6) \]
which of course gives
\[ \hQ(n) = 2 \sum_{(c')^6 \mid n} \abs{\mu(c)} \twisth(n/(c')^6) \]
proving (a).  
	
	For (b), we see that $h_7(n)$ is the the $n$th coefficient of the Dirichlet convolution of $\twistL(s)$ and 
	\[
	2\sum_{n \geq 1} \abs{\mu(n)} n^{-6s} = \frac{2\zeta(6s)}{\zeta(12s)}.
	\]
	Write $s = \sigma + i t$. As both $\twistL(s)$ and $\zeta(6s)/\zeta(12s)$ are absolutely convergent when $\sigma > 1/6$, we see 
	\[
	\LQ(s) = \frac{2 \zeta(6s) \twistL(s)}{\zeta(12s)}
	\]
	when $\sigma > 1/6$, and $\LQ(s)$ converges absolutely in this half-plane.
	
	For (c), since $\zeta(s)$ is nonvanishing when $\sigma > 1$, the ratio $\zeta(6s)/\zeta(12s)$ is meromorphic for $\sigma > 1/12$. But \Cref{Corollary: twL(s) has a meromorphic continuation} gives a meromorphic continuation of $\twistL(s)$ to the region \eqref{Equation: domain of twistL(s)}. The function $\LQ(s)$ is a product of these two meromorphic functions on \eqref{Equation: domain of twistL(s)}, and so it is a meromorphic function on this region. The holomorphy and singularity for $\LQ(s)$ then follow from those of $\twistL(s)$ and $\zeta(s)$. 
	
	We conclude (d) by computing Laurent expansions. We readily verify
\begin{equation}
\frac{\zeta(6s)}{\zeta(12s)} = \frac{1}{\zeta(2)}\parent{\frac{1}{6}\parent{s - \frac 16}\inv + \parent{\gamma - \frac{2 \zeta\prm(2)}{\zeta(2)}} + O\parent{s - \frac 16}},
\end{equation}
whereas the Laurent expansion for $\twistL(s)$ at $s = 1/6$ begins
\begin{equation}
\twistL(s) = \frac{1}{\zeta(2)} \parent{\frac{QR}{6}\parent{s - \frac 16}\inv + \zeta(2) \ell_0 + \dots},
\end{equation}
with $\ell_0$ given by \eqref{Equation: Definition of ell0}. Multiplying the Laurent series tails gives the desired result.
\end{proof}

\subsection{Proof of main result}

We are now poised to finish off the proof of our main result.

\begin{lemma}\label{Lemma: h(n) is admissible}
	The sequence $\parent{\hQ(n)}_{n \geq 1}$ is admissible \textup{(\Cref{Definition: admissible sequences})} with parameters $(1/6,1/30,128/1025)$.
\end{lemma}

\begin{proof}
	We check each condition in \Cref{Definition: admissible sequences}. Since $\hQ(n)$ counts objects, 
we indeed have $\hQ(n) \in \Z_{\geq 0}$.
	
	For (i), \Cref{Corollary: twL(s) has a meromorphic continuation} tells us that $\twistL(s)$ has $1/6$ as its abscissa of absolute convergence. Likewise, $\displaystyle{\frac{\zeta(6s)}{\zeta(12s)}}$ has $1/6$ as its abscissa of absolute convergence. By \Cref{Theorem: relationship between twistL(s) and L(s)}(b), 
	\[
	\LQ(s) = \frac{2\zeta(6s) \twistL(s)}{\zeta(12s)},
	\]
	and by \Cref{Theorem: product of Dirichlet series converges} this series converges absolutely for $\sigma > \sigma_a$, so the abscissa of absolute convergence for $\LQ(s)$ is at most $1/6$. But for $\sigma < 1/6$, $\LQ(\sigma) > \twistL(\sigma)$ by termwise comparison of coefficients, so the Dirichlet series for $\LQ(s)$ diverges when $\sigma < 1/6$, and (i) holds with $\sigma_a = 1/6$.

For (ii), \Cref{Corollary: twL(s) has a meromorphic continuation} tells us that $\twistL(s)$ has a meromorphic continuation when $\sigma=\repart(s)>2/15$; on the other hand, as $\zeta(12s)$ is nonvanishing for $\sigma > 1/12$, we see that $\zeta(6s)/\zeta(12s)$ has a meromorphic contintuation to $\sigma>1/12$, and so (ii) holds with 
\[
\delta = 1/6 - 2/15=1/30.
\]
(The only pole of $\LQ(s)/s$ with $\sigma > 2/15$ is the double pole at $s = 1/6$ indicated in 
\Cref{Theorem: relationship between twistL(s) and L(s)}(e).)
	
For (iii), let $\sigma > 2/15$. Let $\zeta_a(s)=\zeta(as)$. Applying \Cref{Theorem: muzeta(sigma)}, we have 
\begin{equation}\label{Equation: bound on zeta6(s)}
\mu_{\zeta_6}(\sigma) = \mu_{\zeta}(6\sigma) < \frac{64}{205}\parent{1 - \frac{6 \cdot 2}{15}} = \frac{64}{1025}
\end{equation} 
if $\sigma \leq 1/6$, and by \Cref{Theorem: absolutely convergent Dirichlet series have mu = 0}, $\mu_{\zeta_6}(\sigma) = 0$ if $\sigma > 1/6$. We recall \eqref{Equation: Defining twLR}; applying \Cref{Theorem: absolutely convergent Dirichlet series have mu = 0} again we see $\mu_{\twistLR}(\sigma) = 0$ if $\sigma > 2/15$, so \eqref{Equation: bound on zeta6(s)} implies $\mu_{\twistL}(\sigma) < 64/1025$ if $\sigma > 2/15$.   Finally, as $\zeta(12s)^{-1}$ is absolutely convergent for $s > 1/12$, \Cref{Theorem: absolutely convergent Dirichlet series have mu = 0} tells us $\mu_{{\zeta_{12}}^{-1}}(\sigma) = 0$. Taken together, we see
	\begin{equation}
	\mu_{\LQ}(\sigma) < \frac{64}{1025} + \frac{64}{1025} + 0 = \frac{128}{1025},
	\end{equation}
	so the sequence $\parent{\hQ(n)}_{n \geq 1}$ is admissible with final parameter $\xi = 128/1025$.
\end{proof}

We now prove \Cref{Intro Theorem: asymptotic for N(X)}, which we restate here for ease of reference in our established notation.

\begin{theorem}\label{Theorem: asymptotic for N(X)}
	For all $\epsilon > 0$, 
		\[
	\NQ(X) = \frac{QR}{3 \zeta(2)^2} X^{1/6} \log X + \frac{2}{\zeta(2)^2}\parent{\zeta(2) \ell_0 + Q R \parent{\gamma - 1 - \displaystyle{\frac{2 \zeta\prm(2)}{\zeta(2)}}}} X^{1/6} + O\parent{X^{3/20 + \epsilon}}
	\]
	as $X \to \infty$, where the implicit constant depends on $\epsilon$. The constants $Q,R$ are defined in \textup{\Cref{Theorem: asymptotic for twN(X)}}, and $\ell_0$ is defined in \eqref{Equation: Definition of ell0}.
\end{theorem}

\begin{proof}
	By \Cref{Lemma: h(n) is admissible}, $\parent{\hQ(n)}_{n \geq 1}$ is admissible with parameters $\parent{1/6, 1/30, 128/1025}$. We now apply \Cref{Theorem: Landau's Tauberian theorem} to the Dirichlet series $\LQ(s)$, and our claim follows. 
\end{proof}

\begin{remark}\label{Remark: true bound for twistN(X)}
We suspect that the true error on both $\NQ(X)$ and $\twistN(X)$ is $O(X^{1/12 + \epsilon})$.  Some improvements to our error term are possible.  Improvements to the error term for $\twistN(X)$ will directly improve the error term for $\NQ(X)$. In addition, we believe that (with appropriate hypotheses) the denominator $\floor{\xi} + 2$ in the exponent of the error for \Cref{Theorem: Landau's Tauberian theorem} can be replaced with $\xi + 1$. If so, the exponent $3/20 + \epsilon$ in the error term may be replaced with $158/1153 + \epsilon$.  Under this assumption, improvements in the estimate of $\mu_\zeta(\sigma)$ will translate directly to improvements in the error term of $\NQ(X)$ (see Bourgain \cite[Theorem 5]{Bourgain}). In the most optimistic scenario, if the Lindel\"{o}f hypothesis holds, the exponent of our error term would be the same as that of $\twistN(X)$.
\end{remark}

\begin{remark}
Here we combine Landau's Tauberian theorem (\Cref{Theorem: Landau's Tauberian theorem}) with \Cref{Theorem: relationship between twistL(s) and L(s)}(b) in order to obtain asymptotics for $\LQ (s)$. In doing so, we implicitly invoke the apparatus of complex analysis, which is used in the proof of Perron's formula and of Landau's Tauberian theorem.  Indeed, this suggests a general strategy.  However, we believe an elementary argument applying Dirichlet's hyperbola method \cite[Theorem I.3.1]{Tenenbaum} to \Cref{Theorem: relationship between twistL(s) and L(s)}(a) could achieve similar asymptotics, and perhaps even modestly improve on the error term. 
\end{remark}

\section{Computations}\label{Section: Computations}

In this section, we conclude by describing some computations which make our main theorems completely explicit.

\subsection{Computing elliptic curves with $7$-isogeny}

%To 10^6 in 0.1137235164642334 seconds.
%To 10^12 in 0.9791388511657715 seconds.
%To 10^18 in 9.698534727096558 seconds.
%To 10^24 in 97.599848985672 seconds.
%To 10^30 in 1049.163053035736 seconds.
%To 10^36 in 10694.673179388046 seconds.
%To 10^42 in 125288.32260918617 seconds.
We begin by outlining an algorithm for computing all elliptic curves that admit a 7-isogeny up to twist height $X$. In a nutshell, we iterate over possible factorizations $e^3 m$ with $m$ cubefree to find all groomed pairs $(a, b)$ for which $C(a, b) = e^3 m$, then check if $\twistheight(A(a, b), B(a, b)) \leq X$.

In detail, our algorithm proceeds as follows. 
\begin{enumalg}
	\item We list all primes $p \equiv 1 \pmod 3$ up to $(X/108)^{1/6}$
	(this bound arises from \Cref{Theorem: Controlling size of twist minimality defect}(a)).
	\item For each pair $(a, b) \in \bbZ^2$ with $b > 0$, $\gcd(a, b) = 1$, $b > 0$, and $C(a, b)$ coprime to $3$ and less than $Y$, we compute $C(a, b)$. We organize the results into a lookup table, so that for each $c$ we can find all pairs $(a, b)$ with $b > 0$, $\gcd(a, b) = 1$, $b > 0$, and $C(a, b) = c$. We append $1$ to our table with lookup value $(1, 0)$. For each $c$ in our lookup table, we record whether $c$ is cubefree by sieving against the primes we previously computed.
	\item For positive integer pairs $(e_0, m)$, $e_0^{12} m^6 \leq X/108$, and $m$ cubefree, we find all groomed pairs $(a, b) \in \bbZ^2$ with $C(a, b) = e_0^3 m$. If $\gcd(e_0, 3) = \gcd(m, 3) = 1$, we can do this as follows. If $e_0^3 < Y$, we iterate over groomed pairs $(a_e, b_e)$ and $(a_m, b_m)$ yielding $C(a_e, b_e) = e_0^3$ and $C(a_m, b_m) = m$ respectively, and taking the product 
	\[
	(a_e + b_e\parent{-1+3\zeta}) (a_m + b_m\parent{-1+3\zeta}) = a + b \parent{-1+3\zeta} \in \bbZ[3\zeta]
	\]
	as in the proof of \Cref{Lemma: c(m) is multiplicative etc}. If $e_0^3 > Y$, we iterate over groomed pairs $(a_e\prm, b_e\prm)$ with $C(a_e\prm, b_e\prm) = e_0$ instead of over groomed pairs $(a_e, b_e)$, and compute 
	\[
	(a_e\prm + b_e\parent{-1+3\zeta})^3 (a_m + b_m\parent{-1+3\zeta}) = a + b \parent{-1+3\zeta} \in \bbZ[3\zeta].
	\]
	If $\gcd(e_0, 3) > 1$ or $\gcd(m, 3) > 1$, we perform the steps above for the components of $e_0$ and $m$ coprime to $3$, and then postmultiply by those groomed pairs $(a_3, b_3) \in \bbZ^2$ with $C(a_3, b_3)$ an appropriate power of $3$ (which is necessarily $9$ or $27$ by \Cref{Lemma: c(m) is multiplicative etc}(b).
	\item For each pair $(a, b)$ with $C(a, b) = e_0^3 m$, obtained in the previous step, we compute $\rawheight(A(a, b), B(a, b))$. We compute the $3$-component of the twist minimality defect $e_3$, the $7$-component of the twice minimality defect $e_7$, and thereby compute the twist minimality defect $e = \lcm(e_0, e_3, e_7)$. We compute the twist height using the reduced pairs $(A(a, b) / e^2, \abs{B(a, b)} / e^3)$. If this result is less than or equal to $X$, we report $(a, b)$, together with their twist height and any auxiliary information we care to record.
\end{enumalg}

We list the first few twist minimal elliptic curves admitting a $7$-isogeny in Table \ref{table:firstfew7}.

\jvtable{table:firstfew7}{
\rowcolors{2}{white}{gray!10}
\begin{tabular}{c c|c c}
$(A, B)$ & $(a, b)$ & $\twistheight(E)$ & $\twistdefect(E)$ \\
\hline\hline
$(-3, 62)$ & $(14, 5)$ & $103788$ & $1029$ \\
$(13, 78)$ & $(21, 4)$ & $164268$ & $1029$ \\
$(37, 74)$ & $(42, 1)$ & $202612$ & $1029$ \\
$(-35, 98)$ & $(0, 1)$ & $259308$ & $21$ \\
$(45, 18)$ & $(35, 2)$ & $364500$ & $1029$ \\
$(-43, 166)$ & $(7, 13)$ & $744012$ & $3087$ \\
$(-75, 262)$ & $(-7, 8)$ & $1853388$ & $1029$ \\
$(-147, 658)$ & $(-56, 1)$ & $12706092$ & $1029$ \\
$(-147, 1582)$ & $(7, 6)$ & $67573548$ & $343$ \\
$(285, 2014)$ & $(28, 3)$ & $109517292$ & $343$ \\
$(-323, 2242)$ & $(-21, 10)$ & $135717228$ & $1029$ \\
$(-395, 3002)$ & $(-63, 2)$ & $246519500$ & $1029$ \\
$(-155, 3658)$ & $(21, 11)$ & $361286028$ & $1029$ \\
$(357, 5194)$ & $(7, 1)$ & $728396172$ & $21$ \\
$(-595, 5586)$ & $(-14, 1)$ & $842579500$ & $63$ \\
$(285, 5662)$ & $(91, 1)$ & $865572588$ & $1029$ \\
$(-603, 5706)$ & $(-28, 11)$ & $879077772$ & $1029$ \\
\end{tabular}
}{$E \in \twistE$ with $7$-isogeny and $\twht E \leq 10^{9}$}

Running this algorithm out to $X = 10^{42}$ took us approximately $34$ CPU hours on a single core, producing 4\,582\,079 
elliptic curves admitting a 7-isogeny in $\twistE_{\leq 10^{42}}$. To check the accuracy of our code, we confirmed that the $j$-invariants of these curves are distinct. We also confirmed that the 7-division polynomial of each curve has a linear or cubic factor over $\bbQ$; this took 3.5 CPU hours. For $X = 10^{42}$, we have
\[
\frac{\zeta(2)}{Q R}\frac{ \twistN(10^{42})}{(10^{42})^{1/6}} = 0.99996\ldots,
\]
which is close to 1.

Substituting \Cref{Theorem: relationship between twistL(s) and L(s)}(a) into \eqref{Equation: NQ(X) in terms of h7} and reorganizing the resulting sum, we find 
\begin{equation}
\NQ (X) = 2 \sum_{n \leq X} \twisth\parent{n/c^6} \sum_{c \leq (X/n)^{1/6}} \abs{\mu(c)}.
\end{equation}
Letting $X = 10^{42}$ and using our list of 4\,582\,079  elliptic curves admitting a 7-isogeny, we compute that there are $88\,157\,174$ elliptic curves admitting a 7-isogeny in $\scrE_{\leq 10^{42}}$.

\subsection{Computing constants}

We also estimate the constants in our main theorems. First and easiest among these is $Q$, given by \eqref{Equation: product for Q}. Truncating the Euler product as a product over $p \leq Y$ gives us a lower bound
\[
Q_{\leq Y} \colonequals \frac{273}{16} \prod_{\substack{7 < p \leq Y \\ p \equiv 1 \psmod 3}} \parent{1 + \frac{2}{p^2+1}}
\]
for $Q$. To obtain an upper bound, we compute
\[
	Q < Q_{\leq Y} \exp\parent{2 \sum_{\substack{p > Y \\ p \equiv 1 \psmod 3}} \frac{1}{p^2+1}}.
\]
For $a, b \in \bbZ$ coprime integers and $X \in \bbR_{>0}$, write
\begin{equation}
\pi(X; a, b) \colonequals \# \set{p \ \text{prime} : p \equiv a \pmod b}.
\end{equation}
Suppose $Y \geq 8 \cdot 10^9$. Using Abel summation and Bennett--Martin--O'Bryant--Rechnitzer \cite[Theorem 1.4]{Bennett-Martin-OBryant-Rechnitzer}, we obtain
\begin{align*}
\sum_{\substack{p > Y \\ p \equiv 1 \psmod 3}} \frac{1}{p^2+1} &= -\frac{\pi(Y;3,1)}{Y^2 + 1} + 2 \int_Y^\infty \frac{\pi(u; 3, 1) u} {(u^2 + 1)^2} \,\mathrm{d}u \\
&< -\frac{Y}{2\parent{Y^2 + 1} \log Y} + \parent{\frac 1{\log Y} + \frac{5}{2 \log^2 Y}} \int_Y^\infty \frac{u^2} {(u^2 + 1)^2} \,\mathrm{d}u \\
&= \frac 12 \parent{\frac{5Y}{2 (Y^2 + 1) \log Y} + \parent{\frac 1{\log Y} + \frac{5}{2 \log^2 Y}} \parent{\frac{\pi}{2} - \tan\inv(Y)}}
\end{align*}
so
\[
Q <Q_{\leq Y} \cdot \exp\parent{\frac{5Y}{2 (Y^2 + 1) \log Y} + \parent{\frac 1{\log Y} + \frac{5}{2 \log^2 Y}} \parent{\frac{\pi}{2} - \tan\inv(Y)}}.
\]
In particular, letting $Y = 10^{12}$, we compute
\[
17.46040523112662 < Q < 17.460405231134835
\]
This computation took approximately $9$ CPU days, although an estimate nearly as good could be computed much more quickly. % 789737 seconds.

We now turn our attention to $R$, given in \eqref{eqn: R(X)}. We observe 
\[
\calR(1) \subseteq [-0.677, 0.677] \times [0, 0.078],
\]
so we can estimate $\calR(1)$ by performing rejection sampling on the rectangle $[-0.677, 0.677] \times [0, 0.078]$, which has area $0.105612$. Of our $s \colonequals 595\,055\,000\,000$ samples, $r \colonequals 243\,228\,665\,965$ lie in $R$, so 
\[
R \approx 0.105612 \cdot\frac{r}{s} = 0.04316889\ldots
\]
with standard error 
\[
0.105612 \cdot \sqrt{\frac{r(s - r)}{s^3}} < 6.8 \cdot 10^{-8}.
\]
This took 11 CPU weeks to compute, although an estimate nearly as good could be computed much more quickly. With a little more care, we believe that $R$ could be estimated via numerical integration with a provable error bound.

We can approximate $\ell_0$ by truncating the integral \eqref{Equation: Definition of ell0} and using our approximations for $Q$ and $R$. This yields $\ell_0 \approx -1.62334$. In \Cref{Theorem: asymptotic for twN(X)}, we have shown that for some $M > 0$ and for all $u > X$, we have
\[
\abs{\twistN(u) - \frac{QR}{\zeta(2)} \floor{u^{1/6}}} < M u^{2/15} \log^{17/5} u.
\]
Thus
\begin{equation*}
\begin{aligned}
&\abs{\int_X^\infty \parent{\twistN(u) - \frac{QR}{\zeta(2)} \floor{u^{1/6}}} u^{-7/6} \,\mathrm{d}u} \\
&\qquad\qquad < M \int_X^\infty u^{-31/30} \log^{17/5} u\,\mathrm{d}u \\
&\qquad\qquad < M \int_X^\infty u^{-31/30} \log^4 u\,\mathrm{d}u \\
&\qquad\qquad = 30 M X^{-1/30} \parent{\log^4 X + 120 \log^3 X + 10800 \log^2 X + 648000 \log X + 19440000};
\end{aligned}
\end{equation*}
this gives us a bound on our truncation error. 
We do not know the exact value for $M$, but empirically, we find that for $1 \leq u \leq 10^{42}$, 
\[
-3.3119 \cdot 10^{-5} \leq \frac{\twistN(u) - \frac{QR}{\zeta(2)} \floor{u^{1/6}}}{u^{2/15} \log^{17/5} u} \leq 4.3226 \cdot 10^{-6}.
\]
If we assume $M \approx 3.3119 \cdot 10^{-5}$, we find the truncation error for $\ell_0$ is bounded by $253.23$, which catastrophically dwarfs our initial estimate. 

We can do better with stronger assumptions. Suppose for the moment that $\twistN(X) - \frac{QR}{\zeta(2)} X^{1/6} = O(X^{1/12 + \epsilon})$, as we guessed in \Cref{Remark: true bound for twistN(X)}. We let $\epsilon \colonequals 10^{-4}$, and find that for $1 \leq u \leq 10^{42}$,
\[
-1.2174 \leq \frac{\twistN(u) - \frac{QR}{\zeta(2)} \floor{u^{1/6}}}{u^{1/12 + \epsilon}} \leq 0.52272.
\]
If
\[
\abs{\twistN(X) - \frac{QR}{\zeta(2)} X^{1/6}} \leq M X^{1/12 + \epsilon}
\]
for $M \approx 1.2174$, we get an estimated truncation error of $2.43 \cdot 10^{-5}$, which is much more manageable.

Our estimate of $\ell_0$ is also skewed by our estimates of $QR$. An error of $\epsilon$ in our estimate for $QR$ induces an error of 
\[
\frac{\epsilon}{6\zeta(2)} \int_1^X \floor{u^{1/6}} u^{-7/6} \,\mathrm{d}u < \frac{\epsilon}{6\zeta(2)} \int_1^X u^{-1}\,\mathrm{d}u = \frac{\epsilon \log X}{6\zeta(2)}
\] 
in our estimate of $\ell_0$. When $X = 10^{42}$, this gives an additional error of $1.15 \cdot 10^{-5}$, for an aggregate error of $253.23$ or $2.43 \cdot 10^{-5}$, depending on our assumptions.

Given $Q$, $R$, and $\ell_0$, it is straightforward to compute $c_1$ and $c_2$ using the expressions given for them in \Cref{Theorem: asymptotic for N(X)}. We find $c_1 = 0.09285536\ldots$ with an error of $6.02 \cdot 10^{-8}$, and $c_2 \approx -0.16405$ with an error of $307.89$ or of $2.98 \cdot 10^{-5}$, depending on the assumptions made above. Note that both of these estimates for $c_2$ depended on empirical rather than theoretical estimates for the implicit constant in the error term of \Cref{Theorem: asymptotic for N(X)}. As a sanity check, we also verify that 
\[
\frac{\twistN(10^{42})}{10^7} - 42 c_1 \log 10 = -0.1641924\ldots \approx c_2,
\]
which agrees to three decimal places with the estimate for $c_2$ we gave above.

\providecommand{\MR}{\relax\ifhmode\unskip\space\fi MR }
% \MRhref is called by the amsart/book/proc definition of \MR.
\providecommand{\MRhref}[2]{%
 \href{http://www.ams.org/mathscinet-getitem?mr=#1}{#2}
}
\providecommand{\href}[2]{#2}


\begin{thebibliography}{10}

\bibitem{Bennett-Martin-OBryant-Rechnitzer}
Michael A.\ Bennett, Greg Martin, Kevin O'Bryant, and Andrew Rechnitzer, \emph{Explicit bounds for primes in arithmetic progressions}, Illinois J. Math. \textbf{62} (2018), no.~1-4, 427--532.

\bibitem{Bingham-Goldie-Teugels}
N.~H. Bingham, C.~M. Goldie, and J.~L. Teugels, \emph{Regular variation},
 Encyclopedia Math. Appl., vol.~27, Cambridge
 University Press, Cambridge, 1987. %\MR{898871}

\bibitem{Boggess-Sankar}
Brandon Boggess and Soumya Sankar, \emph{Counting elliptic curves with a
 rational $n$-isogeny for small $n$}, 2020, \texttt{arXiv:2009.05223}.

\bibitem{Bourgain}
J. Bourgain, \emph{Decoupling, exponential sums and the Riemann zeta function}, J.
Amer. Math. Soc. 30 (2017), no. 1, 205–224. %MR 3556291

\bibitem{Bruin-Najman}
Peter Bruin and Filip Najman, \emph{Counting elliptic curves with prescribed
 level structures over number fields}, J. Lond. Math. Soc. (2) \textbf{105}
 (2022), no.~4, 2415--2435. %\MR{4440538}

\bibitem{Cullinan-Kenney-Voight}
John Cullinan, Meagan Kenney, and John Voight, \emph{On a probabilistic
 local-global principle for torsion on elliptic curves}, J. Th\'{e}or. Nombres
 Bordeaux \textbf{34} (2022), no.~1, 41--90. %\MR{4450609}

\bibitem{Davenport}
H.~Davenport, \emph{On a principle of {L}ipschitz}, J. London Math. Soc.
 \textbf{26} (1951), 179--183. %\MR{43821}

\bibitem{Diamond-Shurman}
Fred Diamond and Jerry Shurman, \emph{A first course in modular forms}, Graduate
Texts in Mathematics, vol. 228, Springer-Verlag, New York, 2005. % MR 2112196

\bibitem{Duke}
William Duke, \emph{Elliptic curves with no exceptional primes}, C. R. Acad.
 Sci. Paris S\'{e}r. I Math. \textbf{325} (1997), no.~8, 813--818.
 %\MR{1485897}

\bibitem{Grant}
David Grant, \emph{A formula for the number of elliptic curves with exceptional
 primes}, Compositio Math. \textbf{122} (2000), no.~2, 151--164. %\MR{1775416}

\bibitem{Harron-Snowden}
R.~Harron and A.~Snowden, \emph{Counting elliptic curves with prescribed
 torsion}, J. Reine Angew. Math. \textbf{729} (2017), 151--170. %\MR{3680373}

\bibitem{Huxley}
M.~N.\ Huxley, \emph{Exponential sums and lattice points. {III}}, 
Proc.\ London Math.\ Soc.\ (3) \textbf{87} (2003), 591--609. %\MR{2005876}

\bibitem{Ivic}
Aleksandar Ivi\'{c}, \emph{The {R}iemann zeta-function}, Dover Publications,
 Inc., Mineola, 2003. %\MR{1994094}

\bibitem{Landau1915}
E.~Landau, \emph{\"{U}ber die Anzahl der Gitterpunkte in gewissen Bereichen.
 (Zweite Abhandlung)}, Nachrichten von der Gesellschaft der Wissenschaften zu
 G\"ottingen, Mathematisch-Physikalische Klasse \textbf{1915} (1915), 209--243.

\bibitem{Roux}
{M}athieu {R}oux, \emph{Th\'{e}orie de l'information, s\'{e}ries de
 {Dirichlet}, et analyse d'algorithmes}, Ph.D.\ thesis, Universit\'e de Caen, 2011.

\bibitem{Phillips1}
Tristan Phillips, \emph{Rational points of bounded height on some genus zero
 modular curves over number fields}, 2022, \texttt{arXiv:2201.10624}.

\bibitem{Pizzo-Pomerance-Voight}
Maggie Pizzo, Carl Pomerance, and John Voight, \emph{Counting elliptic curves
 with an isogeny of degree three}, Proc. Amer. Math. Soc. Ser. B \textbf{7}
 (2020), 28--42. %\MR{4071798}

\bibitem{Pomerance-Schaefer}
Carl Pomerance and Edward~F. Schaefer, \emph{Elliptic curves with
 {G}alois-stable cyclic subgroups of order 4}, Res. Number Theory \textbf{7}
 (2021), no.~2, Paper No. 35, 19 pages. %\MR{4256691}

\bibitem{RSZB}
Jeremy Rouse, Andrew V.\ Sutherland, and David Zureick-Brown, appendix with John Voight, \emph{$\ell$-adic images of Galois for elliptic curves over $\Q$}, Forum Math.\ Sigma \textbf{10} (2022), e62.

\bibitem{Silverman}
Joseph~H. Silverman, \emph{The arithmetic of elliptic curves}, second ed.,
 Grad. Texts in Math., vol.\ 106, Springer, Dordrecht, 2009.
 %\MR{2514094}

\bibitem{Tenenbaum}
G\'{e}rald Tenenbaum, \emph{Introduction to analytic and probabilistic number
 theory}, third ed., Grad. Studies in Math., vol. 163, American
 Mathematical Society, Providence, RI, 2015. %
 % Translated from the 2008 French edition by Patrick D. F. Ion. %\MR{3363366}

\bibitem{Widder}
David~Vernon Widder, \emph{The {L}aplace transform}, Princeton Math.\
 Ser., vol. 6, Princeton University Press, Princeton, 1941.
 %\MR{0005923}

\bibitem{Zhai}
Wenguang~Zhai, \emph{{A}symptotics for a class of arithmetic functions}, Acta Arith. \textbf{2} (2015), 135--160.
%https://eudml.org/doc/279346

\end{thebibliography}
\end{document}